\documentclass[12pt]{article}
\usepackage{amssymb,amsmath,amsthm, amsfonts}
\usepackage{graphicx}
\usepackage{epsfig}
\usepackage{tikz}
\usepackage{mathrsfs}

\def\disp{\displaystyle}

\def\dref#1{(\ref{#1})}

\theoremstyle{plain}
\newtheorem{theorem}{Theorem}[section]
\newtheorem{lemma}{Lemma}[section]

\theoremstyle{definition}
\newtheorem{definition}{Definition}[section]

\newtheorem{remark}{Remark}[section]

\numberwithin{equation}{section}

\begin{document}

\title{\bf A new approach toward locally bounded global solutions to a  $3D$ chemotaxis-stokes system with  nonlinear diffusion and rotation}

\author{
Jiashan Zheng\thanks{Corresponding author.   E-mail address:
 zhengjiashan2008@163.com (J.Zheng)}
 \\
    School of Mathematics and Statistics Science,\\
     Ludong University, Yantai 264025,  P.R.China \\
}
\date{}


\maketitle \vspace{0.3cm}
\noindent
\begin{abstract}
We consider a degenerate quasilinear chemotaxis--Stokes
type involving rotation in the aggregative term,
$$
 \left\{
 \begin{array}{l}
   n_t+u\cdot\nabla n=\Delta n^m-\nabla\cdot(nS(x,n,c)\cdot\nabla c),\quad
x\in \Omega, t>0,\\
    c_t+u\cdot\nabla c=\Delta c-nc,\quad
x\in \Omega, t>0,\\
u_t+\nabla P=\Delta u+n\nabla \phi ,\quad
x\in \Omega, t>0,\\
\nabla\cdot u=0,\quad
x\in \Omega, t>0,\\
 \end{array}\right.\eqno(CF)
 $$
 where $\Omega\subseteq \mathbb{R}^3$ is a bounded  convex  domain with smooth boundary.
 Here
$ S\in C^2(\bar{\Omega}\times[0,\infty)^2;\mathbb{R}^{3\times3})$ is a matrix with $s_{i,j}\in  C^1( \bar{\Omega} \times [0, \infty)\times[0, \infty)).$
Moreover, $|S(x,n,c)| \leq  S_0(c)$ for all $(x,n,c)\in \bar{\Omega} \times [0, \infty)\times[0, \infty)$ with $S_0(c)$
nondecreasing on $[0,\infty)$.
 If
$$m>\frac{9}{8},
$$
then
for all reasonably regular initial data, a corresponding initial-boundary value
problem for $(CF)$
possesses a globally defined weak solution $(n,c,u)$. Moreover,  for any fixed $T > 0$ this solution is bounded in $\Omega\times  (0,T)$ in the sense that
$$
\|u(\cdot,t)\|_{L^\infty(\Omega)} +\|c(\cdot,t)\|_{W^{1,\infty}(\Omega)}+\|n(\cdot,t)\|_{L^\infty(\Omega)} \leq C ~~\mbox{for all}~~ t\in(0,T)
$$
is valid with some $C(T) > 0$.
In particularly, if $ S(x,n,c):=C_S$,
 this result extends of Tao and  Winkler (\cite{Tao71215}), while,
   if  fluid-free subcase of
   the flow of fluid is ignored or the fluid is
stationary in $(CF)$, $ S(x,n,c):=C_S$ and $N=3$,  this results
is consistent with the result of Theorem 2.1 of  Zheng and Wang  (\cite{Zhengddffsdsd6}).
In view of  some carefully analysis,
we can  establish some natural gradient-like structure
 of the
functional
$
\int_{\Omega}n(\cdot,t)\ln n(\cdot,t)+\int_{\Omega}|\nabla\sqrt{c}(\cdot,t)|^2+\int_{\Omega}|u(\cdot,t)|^2
$
of $(CF),$
 which is
 a new estimate of chemotaxis--Stokes system with rotation (see
 \cite{Bellomo1216,Cao22119,Winkler11215,Wang21215}).

\end{abstract}

\vspace{0.3cm}


\newpage
\section{Introduction}

We consider the following  chemotaxis-Stokes system  with porous medium
diffusion and rotation in the aggregation term:
\begin{equation}
 \left\{\begin{array}{ll}
   n_t+u\cdot\nabla n=\nabla\cdot(D(n)\nabla n) -\nabla\cdot( nS(x,n,c)\cdot\nabla c),\quad
x\in \Omega, t>0,\\
    c_t+u\cdot\nabla c=\Delta c-nc,\quad
x\in \Omega, t>0,\\
u_t+\nabla P=\Delta u+n\nabla \phi ,\quad
x\in \Omega, t>0,\\
\nabla\cdot u=0,\quad
x\in \Omega, t>0,\\
 \disp{(\nabla n-nS(x,n,c)\cdot\nabla c)\cdot\nu=\nabla c\cdot\nu=0,u=0,}\quad
x\in \partial\Omega, t>0,\\
\disp{n(x,0)=n_0(x),c(x,0)=c_0(x),u(x,0)=u_0(x),}\quad
x\in \Omega,\\
 \end{array}\right.\label{1.1}
\end{equation}
where
$ \Omega$ is a bounded  convex  domain in $\mathbb{R}^3$ 
 with smooth boundary
$\partial\Omega$,
$S(x, n, c)$ is a
chemotactic sensitivity tensor satisfying
\begin{equation}\label{x1.73142vghf48rtgyhu}
S\in C^2(\bar{\Omega}\times[0,\infty)^2;\mathbb{R}^{3\times3})
 \end{equation}
 and
 \begin{equation}\label{x1.73142vghf48gg}|S(x, n, c)|\leq  S_0(c) ~~~~\mbox{for all}~~ (x, n, c)\in\Omega\times [0,\infty)^2
 \end{equation}
with  some nondecreasing $S_0 : [0,\infty)\rightarrow \mathbb{R},$
 porous medium
diffusion function $D$ satisfies
\begin{equation}\label{ghnjmk9161gyyhuug}
D\in C^\iota_{loc} ([0,\infty))~~\mbox{for some}~~\iota>0,~~~
D(n)\geq C_Dn^{m-1}~~ \mbox{for all}~~ n>0
\end{equation}
with some $m \geq 1$.
Here $n$ and $c$ denote the bacterium density and the oxygen concentration, respectively,
and $u$ represents the velocity field of the
fluid subject to an incompressible
Navier-Stokes equation with pressure $P$ and viscosity $\eta$ and a gravitational force
$\nabla\phi$.
This type of the system arises in mathematical biology to model the evolution of oxygen-driven
swimming bacteria in an impressible fluid.

To motivate our study, let us first recall the following fluid-free subcase of system \dref{1.1}:
\begin{equation}
 \left\{\begin{array}{ll}
   n_t=\nabla\cdot(D(n)\nabla n) -\nabla\cdot( nS(x,n,c)\cdot\nabla c),\quad
x\in \Omega, t>0,\\
    c_t=\Delta c-nc
    ,\quad
x\in \Omega, t>0,\\
 \disp{\nabla n\cdot\nu=\nabla c\cdot\nu=0,}\quad
x\in \partial\Omega, t>0,\\
\disp{n(x,0)=n_0(x),c(x,0)=c_0(x),}\quad
x\in \Omega.\\
 \end{array}\right.\label{hjmkl1.1}
\end{equation}
There are only few rigorous results on global existence and qualitative behavior
of solutions  to \dref{hjmkl1.1}  with  either a matrix-valued function ($S(x,n,c)$)
or a scalar one $(S(x,n,c):=S(c))$ (see e.g. \cite{Tao2,Zhengddffsdsd6,Winklersedf3411215}).

 The following chemotaxis-(Navier)-Stokes
model
 which is a generalized version of the model proposed in \cite{Tuval1215}, describes the motion
of oxygen-driven swimming cells in an incompressible fluid:
%
%
\begin{equation}
 \left\{\begin{array}{ll}
   n_t+u\cdot\nabla n=\nabla\cdot(D(n)\nabla n) -C_S\nabla\cdot( nS(c)\cdot\nabla c),\quad
x\in \Omega, t>0,\\
    c_t+u\cdot\nabla c=\Delta c-nc,\quad
x\in \Omega, t>0,\\
u_t+(u\cdot\nabla u)=\nabla P+\Delta u+n\nabla \phi ,\quad
x\in \Omega, t>0,\\
\nabla\cdot u=0,\quad
x\in \Omega, t>0,\\
 \end{array}\right.\label{ghhbnmkl1.1}
\end{equation}
where compared with \dref{1.1}, the nonlinear convective term $u\nabla u$  exists
$u$-equation of \dref{ghhbnmkl1.1}, moreover $S(x,n,c):=S(c)$ is scalar function in \dref{ghhbnmkl1.1} and hence there is a certain natural quasi-Lyapunov functionals of \dref{ghhbnmkl1.1}. Hence,
by making use of energy-type functionals, some
local and global solvability of corresponding initial value problem for \dref{ghhbnmkl1.1} in either bounded or unbounded domains have
been obtained in the past years (see e.g.
Lorz et al. \cite{Duan12186,Liucvb12176},
Winkler et al. \cite{Bellomo1216,Tao795,Winklercvb12176}, Chae et al. \cite{Chaex12176,Chaexdd12176}, Di Francesco et al. (\cite{Francesco791},
Zhang, Zheng \cite{Zhangcvb12176} and references therein).

If the chemotactic sensitivity $S(x, n, c)$ is regarded as a tensor  rather than a scalar one (\cite{Xue1215}), \dref{1.1}
turns into a chemotaxis--Stokes system with rotational
flux which implies that chemotactic migration
need not be directed along the gradient of signal concentration. In contrast to the chemotaxis-fluid system
\dref{ghhbnmkl1.1}, chemotaxis-fluid systems with tensor-valued sensitivity lose some natural gradient-like structure (see Cao \cite{Cao22119}, Wang et. al \cite{Wang11215,Wang21215},
 Winkler \cite{Winkler11215}). This gives rise to
considerable mathematical difficulties.
Therefore,
only very few results appear to be available on chemotaxis--Stokes system with such tensor-valued
sensitivities (Cao et al. \cite{Cao22119}, Ishida \cite{Ishida1215}, Wang et al. \cite{Wang11215,Wangdfg11215,Wang21215}, Winkler \cite{Winkler11215}). In fact, 
assuming 
\dref{x1.73142vghf48rtgyhu}--\dref{x1.73142vghf48gg} holds,
 Ishida (\cite{Ishida1215}) showed  that
 the corresponding full chemotaxis Navier-Stokes system
   with porous-medium-type diffusion model  possesses
a bounded global weak solution in two space dimensions.
While, in three space dimensions,  if the initial data satisfy certain {\bf smallness}
conditions and $D(n)=1$,
 Cao and Lankeit \cite{Caoddff22119}  showed that \dref{1.1} has global classical solutions and give decay properties of these solutions.
%
In this paper, the core step is to establish the estimates of the
functional
%
\begin{equation}\int_{\Omega}n^{p}(\cdot,t) +\int_{\Omega} |\nabla {c}(\cdot,t)|^{2{q}}
 \label{ghnjjffff1.1hhjjddssggtyy}
\end{equation}
 for suitably chosen but arbitrarily large numbers $p > 1$ and $q > 1$.
 In fact, one of  our main tool is consideration of the  natural gradient-like energy functional
%
%
\begin{equation}
\int_{\Omega}n_{\varepsilon}(\cdot,t)\ln n_{\varepsilon}(\cdot,t)+\int_{\Omega}|\nabla\sqrt{c_{\varepsilon}}(\cdot,t)|^2+\int_{\Omega}|u_{\varepsilon}(\cdot,t)|^2,\label{ghbbnnnnmmnjjffff1.1hhjjddssggtyy}
\end{equation}
 which is
  new estimate of chemotaxis--Stokes system   {\bf with rotation} (see Lemmata \ref{lemma630jklhhjj}--\ref{lemmakkllgg4563025xxhjklojjkkk}),
   although, \dref{ghbbnnnnmmnjjffff1.1hhjjddssggtyy} has been used to solve the chemotaxis-(Navier)-Stokes system  {\bf without rotation} (see \cite{Bellomo1216,Lankeitffg11,Winklercvb12176}).
 Here $(n_\varepsilon,c_{\varepsilon},u_{\varepsilon})$ is solution of the approximate problem of \dref{1.1}. We guess that \dref{ghbbnnnnmmnjjffff1.1hhjjddssggtyy} can also be dealt with other types of systems,  e.g., quasilinear chemotaxis
system with rotation, chemotaxis-(Navier)-Stokes
system with rotation.
 Then, in view of the estimates \dref{ghbbnnnnmmnjjffff1.1hhjjddssggtyy}, the suitable
interpolation arguments (see Lemma  \ref{drfe116lemma70hhjj}) and the basic a priori information  (see Lemma \ref{ghjssdeedrfe116lemma70hhjj}), we can get the the estimates of the
functional
\begin{equation}
\int_{\Omega}n^{p}_{\varepsilon} +\int_{\Omega} |\nabla {c}_{\varepsilon}|^{2{q_0}}+\int_{\Omega}|A^{\frac{1}{2}}u_{\varepsilon}|^2 ,\label{ghbbnnnnmmnjjffffxcvvvvvvffgg1.1hhjjddssggtyy}
\end{equation}
 where $p:=p(q_0,m)\geq\frac{7}{4}$ and $q_0<2$.
Next, \dref{ghbbnnnnmmnjjffffxcvvvvvvffgg1.1hhjjddssggtyy} and some other carefully analysis  (Lemma \ref{lemma4ggg563025xxhjklojjkkkgyhuirrtt}--\ref{lemma4563025xxhjklojjkkkgyhuissddff})
yield  the estimates of the
functional
 \dref{ghnjjffff1.1hhjjddssggtyy}.
 Indeed,
the paper makes sure that under the assumption $m$  satisfies \dref{ghnjmk9161gyyhuug}, $S$ satisfies \dref{x1.73142vghf48rtgyhu} and \dref{x1.73142vghf48gg}
with some
\begin{equation}\label{x1.73142vghf48}m>\frac{9}{8},
\end{equation}
then  problem \dref{1.1}
possesses a  global weak  solution, which extends the results of Tao and Winkler
(\cite{Tao71215}), who showed the global existence of solutions in
 the cases $S(x,n,c):=C_S$,
 $m$  satisfies \dref{ghnjmk9161gyyhuug} with $m>\frac{8}{7}$.
\section{Preliminaries and  main results}
Due to the hypothesis \dref{ghnjmk9161gyyhuug}, the problem \dref{1.1} has no
classical solutions in general, and thus we consider its weak solutions in the following sense.
\begin{definition} \label{df1} (weak solutions) Let $T\in (0,\infty)$ and  $$H(s)=\int_0^sD(\sigma)d\sigma~~\mbox{for}~~s\geq0.$$
 Suppose that $(n_0, c_0, u_0)$ satisfies \dref{ccvvx1.731426677gg}. 
Then a triple of functions $(n, c, u)$ defined in $\Omega\times(0,T)$ is called a weak solution of model  \dref{1.1},
if
\begin{equation}
 \left\{\begin{array}{ll}
   n\in L_{loc}^1(\bar{\Omega}\times[0,T)),\\
    c \in L_{loc}^\infty(\bar{\Omega}\times[0,T))\cap L_{loc}^1([0,T); W^{1,1}(\Omega)),\\
u \in  L_{loc}^1([0,T); W^{1,1}(\Omega)),\\
 \end{array}\right.\label{dffff1.1fghyuisdakkklll}
\end{equation}
where $n\geq 0$ and $c\geq 0$ in
$\Omega\times(0, T)$ as well as $\nabla\cdot u = 0$ in the distributional sense in
 $\Omega\times(0, T)$,
in addition,
\begin{equation}\label{726291hh}
H(n),~~ n|\nabla c|~~~ \mbox{and}~~~ n|u|~~ \mbox{belong to}~~ L^1_{loc}(\bar{\Omega}\times [0, T)),
\end{equation}
and
\begin{equation}
\begin{array}{rl}\label{eqx45xx12112ccgghh}
\disp{-\int_0^{T}\int_{\Omega}n\varphi_t-\int_{\Omega}n_0\varphi(\cdot,0)  }=&\disp{
\int_0^T\int_{\Omega}H(n)\Delta\varphi+\int_0^T\int_{\Omega}n(S(x,n,c)\cdot
\nabla c)\cdot\nabla\varphi}\\
&+\disp{\int_0^T\int_{\Omega}nu\cdot\nabla\varphi}\\
\end{array}
\end{equation}
for any $\varphi\in C_0^{\infty} (\bar{\Omega}\times[0, T))$ satisfying
 $\frac{\partial\varphi}{\partial\nu}= 0$ on $\partial\Omega\times (0, T)$
  as well as
  \begin{equation}
\begin{array}{rl}\label{eqx45xx12112ccgghhjj}
\disp{-\int_0^{T}\int_{\Omega}c\varphi_t-\int_{\Omega}c_0\varphi(\cdot,0)  }=&\disp{-
\int_0^T\int_{\Omega}\nabla c\cdot\nabla\varphi-\int_0^T\int_{\Omega}nc\cdot\varphi+
\int_0^T\int_{\Omega}cu\cdot\nabla\varphi}\\
\end{array}
\end{equation}
for any $\varphi\in C_0^{\infty} (\bar{\Omega}\times[0, T))$  and
\begin{equation}
\begin{array}{rl}\label{eqx45xx12112ccgghhjjgghh}
\disp{-\int_0^{T}\int_{\Omega}u\varphi_t-\int_{\Omega}u_0\varphi(\cdot,0)  }=&\disp{-
\int_0^T\int_{\Omega}\nabla u\cdot\nabla\varphi-
\int_0^T\int_{\Omega}n\nabla\phi\cdot\varphi}\\
\end{array}
\end{equation}
for any $\varphi\in C_0^{\infty} (\bar{\Omega}\times[0, T);\mathbb{R}^3)$ fulfilling
$\nabla\varphi\equiv 0$ in
 $\Omega\times(0, T)$.
If $(n, c, u)$ is a weak solution of \dref{1.1} in $\Omega\times(0,T )$ for any $T\in (0,\infty),$
then we call $(n, c, u)$ a global weak solution.
\end{definition}

In this paper,
we assume that 
\begin{equation}
\phi\in W^{1,\infty}(\Omega).
\label{dd1.1fghyuisdakkkllljjjkk}
\end{equation}
Moreover, let the initial data
$(n_0, c_0, u_0)$ fulfill
\begin{equation}\label{ccvvx1.731426677gg}
\left\{
\begin{array}{ll}
\displaystyle{n_0\in C^\kappa(\bar{\Omega})~~\mbox{for certain}~~ \kappa > 0~~ \mbox{with}~~ n_0\geq0 ~~\mbox{in}~~\Omega},\\
\displaystyle{c_0\in W^{1,\infty}(\Omega)~~\mbox{with}~~c_0\geq0~~\mbox{in}~~\bar{\Omega},}\\
\displaystyle{u_0\in D(A^\gamma_{r})~~\mbox{for~~ some}~~\gamma\in ( \frac{3}{4}, 1)~~\mbox{and any}~~ {r}\in (1,\infty),}\\
\end{array}
\right.
\end{equation}
where $A_{r}$ denotes the Stokes operator with domain $D(A_{r}) := W^{2,{r}}(\Omega)\cap  W^{1,{r}}_0(\Omega)
\cap L^{r}_{\sigma}(\Omega)$,
and
$L^{r}_{\sigma}(\Omega) := \{\varphi\in  L^{r}(\Omega)|\nabla\cdot\varphi = 0\}$ for ${r}\in(1,\infty)$
 (\cite{Sohr}).

\begin{theorem}\label{theorem3}
Let
\dref{dd1.1fghyuisdakkkllljjjkk}
 hold, and suppose that $m$ and $S$ satisfies \dref{ghnjmk9161gyyhuug} and \dref{x1.73142vghf48rtgyhu}--\dref{x1.73142vghf48gg}, respectively.
 Suppose that the assumptions \dref{ccvvx1.731426677gg} hold.
 If
\begin{equation}\label{x1.73142vghf48}m>\frac{9}{8},
\end{equation}
then there exists at least one global weak
solution (in the sense of Definition \ref{df1} above) of problem \dref{1.1}. Moreover,
for any fixed $T > 0$, there exists a positive constant $C:=C(T)$ such that
 \begin{equation}
\|u(\cdot,t)\|_{L^\infty(\Omega)} +\|c(\cdot,t)\|_{W^{1,\infty}(\Omega)}+\|n(\cdot,t)\|_{L^\infty(\Omega)} \leq C ~~\mbox{for all}~~ t\in(0,T).
\label{zj233455scz2.529711234566x9630111kk}
\end{equation}

\end{theorem}
\begin{remark}
(i) If $ S(x,n,c):=C_S,$   Theorem \ref{theorem3}   extends the results of Theorem 1.1 of Tao and Winkler \cite{Tao71215}, who proved the possibility of $\mathbf{global~~ existence}$,
in the case that $m> \frac{8}{7}$.



(ii) In view of Theorem \ref{theorem3}, if the flow of fluid is ignored or the fluid is
stationary in \dref{1.1}, $ S(x,n,c):=C_S,$ and $N=3$,  Theorem \ref{theorem3}
is consistent with the result of Theorem 2.1 of  Zheng and Wang(\cite{Zhengddffsdsd6}),
who proved the possibility of $\mathbf{global~~ existence}$,
in the case that $m> \frac{9}{8}$.

(iii) From Theorem \ref{theorem3}, if the flow of fluid is ignored or the fluid is
stationary in \dref{1.1} and $ S(x,n,c):=C_S,$ and $N=3$, our results improve of  Wang et al. \cite{Wangdd79k1},
who proved the possibility of  $\mathbf{global~~ existence}$,
in the case that $m> \frac{4}{3}$.

(iv) By Theorem \ref{theorem3}, we also derive that the large diffusion exponent $m$ ($>\frac{10}{9}$) yields  the existence  of
solutions to \dref{1.1}. Moreover, no smallness condition on either $\phi$ or on the initial data needs to be fulfilled here, which is different from \cite{Caoddff22119}.

\end{remark}


\begin{lemma}(\cite{Winkler11215})\label{lemmafggg78630jklhhjj}
 Let $l\in[1,+\infty)$ and $r\in[1,+\infty]$ be such that
 \begin{equation}
\left\{\begin{array}{ll}
l<\frac{3r}{3-r}~~\mbox{if}~~
r\leq3,\\
l\leq\infty~~\mbox{if}~~
r>3.
 \end{array}\right.\label{3.10gghhjuulooll}
\end{equation}
Then for all $K > 0$ there exists $C: = C(l, r,K)$ such that if 
 \begin{equation}\|n{}(\cdot, t)\|_{L^r(\Omega)}\leq K~~ \mbox{for all}~~ t\in(0, T_{max}),
\label{3.10gghhjuuloollgghhhy}
\end{equation}
then
 \begin{equation}\|D u{}(\cdot, t)\|_{L^l(\Omega)}\leq C~~ \mbox{for all}~~ t\in(0, T_{max}).
\label{3.10gghhjuuloollgghhhyhh}
\end{equation}
\end{lemma}

In general, the degenerate diffusion case of \dref{1.1}
might not have classical solutions, thus in order to justify all the formal arguments, we need to introduce the following approximating system of \dref{1.1}:
\begin{equation}
 \left\{\begin{array}{ll}
   n_{\varepsilon t}+u_{\varepsilon}\cdot\nabla n_{\varepsilon}=\nabla\cdot(D_{\varepsilon}(n_{\varepsilon})\nabla n_{\varepsilon})-\nabla\cdot(n_{\varepsilon}S_\varepsilon(x, n_{\varepsilon}, c_{\varepsilon})\nabla c_{\varepsilon}),\quad
x\in \Omega, t>0,\\
    c_{\varepsilon t}+u_{\varepsilon}\cdot\nabla c_{\varepsilon}=\Delta c_{\varepsilon}-n_{\varepsilon}c_{\varepsilon},\quad
x\in \Omega, t>0,\\
u_{\varepsilon t}+\nabla P_{\varepsilon}=\Delta u_{\varepsilon}+n_{\varepsilon}\nabla \phi ,\quad
x\in \Omega, t>0,\\
\nabla\cdot u_{\varepsilon}=0,\quad
x\in \Omega, t>0,\\
 \disp{\nabla n_{\varepsilon}\cdot\nu=\nabla c_{\varepsilon}\cdot\nu=0,u_{\varepsilon}=0,\quad
x\in \partial\Omega, t>0,}\\
\disp{n_{\varepsilon}(x,0)=n_0(x),c_{\varepsilon}(x,0)=c_0(x),u_{\varepsilon}(x,0)=u_0(x)},\quad
x\in \Omega,\\
 \end{array}\right.\label{1.1fghyuisda}
\end{equation}
where
%
%
a family
$(D_\varepsilon)_{\varepsilon\in(0,1)}$ of
functions
$$D_\varepsilon\in C^2((0,\infty))~ \mbox{such that}~~ D_\varepsilon(n)\geq\varepsilon
 ~\mbox{for all}~ n > 0 $$$$~~\mbox{and}~
D(n) \leq D_\varepsilon(n)\leq D(n)  + 2\varepsilon~ \mbox{for all}~
 n > 0 ~\mbox{and}~\varepsilon\in (0, 1),$$
  \begin{equation}
 \begin{array}{ll}
S_\varepsilon(x, n, c) := \rho_\varepsilon(x)S(x, n, c),~~ x\in\bar{\Omega},~~n\geq0,~~c\geq0 ~~\mbox{and}~~\varepsilon\in(0, 1).
 \end{array}\label{3.10gghhjuuloollyuigghhhyy}
\end{equation}
Here $(\rho_\varepsilon)_{\varepsilon\in(0,1)} \in C^\infty_0 (\Omega)$
  be a family of standard cut-off functions satisfying $0\leq\rho_\varepsilon\leq 1$
   in $\Omega$
 and $\rho_\varepsilon\rightarrow1$ in $\Omega$
 as $\varepsilon\rightarrow0$.

Let us begin with the following statement on local well-posedness of \dref{1.1fghyuisda}, along with a convenient
extensibility criterion. For a proof we refer to (see \cite{Winkler11215}, Lemma 2.1 of \cite{Winklercvb12176}):
\begin{lemma}\label{lemma70}
Let $\Omega \subseteq \mathbb{R}^3 $ be a bounded  convex  domain with smooth boundary.
Assume
that
%
the initial data $(n_0,c_0,u_0)$ fulfills \dref{ccvvx1.731426677gg}.
Then there exist $T_{max}\in  (0,\infty]$ and
a classical solution $(n_{\varepsilon}, c_{\varepsilon}, u_{\varepsilon}, P_{\varepsilon})$ of \dref{1.1fghyuisda} in
$\Omega\times(0, T_{max})$ such that
\begin{equation}
 \left\{\begin{array}{ll}
 n_{\varepsilon}\in C^0(\bar{\Omega}\times[0,T_{max}))\cap C^{2,1}(\bar{\Omega}\times(0,T_{max})),\\
  c_{\varepsilon}{}\in  C^0(\bar{\Omega}\times[0,T_{max}))\cap C^{2,1}(\bar{\Omega}\times(0,T_{max})),\\
  u_{\varepsilon}{}\in  C^0(\bar{\Omega}\times[0,T_{max}))\cap C^{2,1}(\bar{\Omega}\times(0,T_{max})),\\
  P_{\varepsilon}{}\in  C^{1,0}(\bar{\Omega}\times(0,T_{max})),\\
   \end{array}\right.\label{1.1ddfghyuisda}
\end{equation}
 classically solving \dref{1.1fghyuisda} in $\Omega\times[0,T_{max})$.
%
Moreover,  $n_\varepsilon$ and $c_\varepsilon$ are nonnegative in
$\Omega\times(0, T_{max})$, and
\begin{equation}
\limsup_{t\nearrow T_{max}}(\|n_\varepsilon(\cdot, t)\|_{L^\infty(\Omega)}+\|c_\varepsilon(\cdot, t)\|_{W^{1,\infty}(\Omega)}+\|A^\gamma u_\varepsilon(\cdot, t)\|_{L^{2}(\Omega)})=\infty,
\label{1.163072x}
\end{equation}
where $\gamma$ is given by \dref{ccvvx1.731426677gg}.
\end{lemma}

\begin{lemma}(\cite{Lankeitffg11})\label{lemmffffgga630jklhhjj}
 Let $w\in C^2(\bar{\Omega})$
  satisfy $\nabla w\cdot\nu  = 0$ on $\partial\Omega$.

    (i) Then
$$\frac{\partial|\nabla w|^2}{\partial\nu} \leq C_{\partial\Omega}|\nabla w|^2,$$
where $C_{\partial\Omega}$ is an upper bound on the curvature of $\partial\Omega$.

(ii) Furthermore, for any $\delta > 0$ there is $C(\delta) > 0$ such that every $w\in C^2(\bar{\Omega})$ with $\nabla w\cdot\nu  = 0$ on $\partial\Omega$ fulfils
$$\|w\|_{L^2(\partial\Omega)}\leq \delta \|\Delta w\|_{L^2(\Omega)} + C(\delta) \|w\|_{L^2(\Omega)} .$$

(iii) For any positive $w\in C^2(\bar{\Omega})$
 \begin{equation}\|\Delta w^{\frac{1}{2}}\|_{L^2(\Omega)}\leq \frac{1}{2} \|w^{\frac{1}{2}}\Delta \ln w\|_{L^2(\Omega)} +
 \frac{1}{4}  \|w^{-\frac{3}{2}}|\nabla w|^2\|_{L^2(\Omega)}.
\label{3.10gghhjuulofffollgghhhyhh}
\end{equation}

(iv) There are $C > 0$ and $\delta > 0$ such that every positive $w\in C^2(\bar{\Omega})$ fulfilling $\nabla w\cdot\nu  = 0$ on $\partial\Omega$ satisfies
\begin{equation}
 \begin{array}{rl}
  &\disp{-2\int_{\Omega}\frac{|\Delta  w|^2 }{ w}+\int_{\Omega}\frac{|\nabla  w|^2\Delta  w }{ w^2}\leq -\delta\int_{\Omega} w|D^2\ln w|^2-\delta\int_{\Omega}\frac{|\nabla  w|^4}{ w^3}+C\int_{\Omega} w.}\\
\end{array}\label{vcbbbbcvvgbhsvvbbsddacvvvvbbqwswddaassffssff3.10deerfgghhjuuloollgghhhyhh}
\end{equation}
\end{lemma}

\begin{lemma}\label{drfe116lemma70hhjj}(Lemma 3.8 of \cite{Winkler11215})
Let $q\geq1$,  \begin{equation}\lambda\in[2q+2,4q+1]
\label{3.10deerfgghhjuuloollgghhhyhh}
\end{equation}
and $\Omega\subset \mathbb{R}^3$ be a bounded  convex  domain with smooth boundary.
Then there exists $C > 0$ such that for all $\varphi\in C^2(\bar{\Omega})$ fulfilling $\varphi\cdot\frac{\partial\varphi}{\partial\nu}= 0$
 on $\partial\Omega$
 we have
 \begin{equation}
 \begin{array}{rl}
 &\|\nabla\varphi\|_{L^\lambda(\Omega)}\leq C\||\nabla\varphi|^{q-1}D^2\varphi\|_{L^2(\Omega)}^{\frac{2(\lambda-3)}{(2q-1)\lambda}}
 \|\varphi\|_{L^\infty(\Omega)}^{\frac{6q-\lambda}{(2q-1)\lambda}}+C\|\varphi\|_{L^\infty(\Omega)}.\\
\end{array}\label{aqwswddaassffssff3.10deerfgghhjuuloollgghhhyhh}
\end{equation}
\end{lemma}

Let us state two well-known results of solution of \dref{1.1fghyuisda}.

\begin{lemma}\label{ghjssdeedrfe116lemma70hhjj}
 The solution  of  \dref{1.1fghyuisda} satisfies
  \begin{equation}
 \begin{array}{rl}
 \|n_{\varepsilon}(\cdot,t)\|_{L^1(\Omega)}=\|n_0\|_{L^1(\Omega)}~~~\mbox{for all}~~t\in (0, T_{max})
\end{array}\label{vgbhssddaqwswddaassffssff3.10deerfgghhjuuloollgghhhyhh}
\end{equation}
and
  \begin{equation}
 \begin{array}{rl}
 \|c_{\varepsilon}(\cdot,t)\|_{L^\infty(\Omega)}\leq\|c_0\|_{L^\infty(\Omega)}~~~\mbox{for all}~~t\in (0, T_{max}).
\end{array}\label{hnjmssddaqwswddaassffssff3.10deerfgghhjuuloollgghhhyhh}
\end{equation}
\end{lemma}

\begin{lemma}\label{lemma630jklhhjj}
For any  $l<\frac{3}{2}$,
there exists $C := C(l, \|n_0\|_{L^1(\Omega)})$ such that
 \begin{equation}\|D u{}(\cdot, t)\|_{L^l(\Omega)}\leq C~~ \mbox{for all}~~ t\in(0, T_{max}).
\label{3.10hhjjgghhjuuloollgghhhyhh}
\end{equation}
\end{lemma}
\begin{proof}
Choosing $r=1$ in Lemma \ref{lemma630jklhhjj} and using \dref{vgbhssddaqwswddaassffssff3.10deerfgghhjuuloollgghhhyhh}, we can get the results.
\end{proof}

\begin{lemma}\label{lemma630jklhhjj}
Let $m>\frac{10}{9}$.
There exists $ C > 0$ independent of $\varepsilon$ such that for every
$\delta_1>0$,
 the solution of \dref{1.1fghyuisda} satisfies
%
%
%
%
\begin{equation}
\int_{\Omega}{|u_{\varepsilon}|^2}+\int_{\Omega}{|\nabla u_{\varepsilon}|^2}\leq\delta_1\int_{\Omega}\frac{D_{\varepsilon}(n_{\varepsilon})|\nabla n_{\varepsilon}|^2}{n_{\varepsilon}}+C~~\mbox{for all}~~ t\in(0, T_{max}),
\label{ddddfgcz2.5ghju48cfg924ghffggyuji}
\end{equation}
\end{lemma}
\begin{proof}
Testing the third equation of \dref{1.1fghyuisda} with $u_\varepsilon$, integrating by parts and using $\nabla\cdot u_{\varepsilon}=0$
\begin{equation}
\frac{1}{2}\int_{\Omega}{|u_{\varepsilon}|^2}+\int_{\Omega}{|\nabla u_{\varepsilon}|^2}= \int_{\Omega}n_{\varepsilon}u_{\varepsilon}\cdot\nabla \phi~~\mbox{for all}~~ t\in(0, T_{max}),
\label{ddddfgcz2.5ghju48cfg924ghyuffggji}
\end{equation}
which together with  the H\"{o}lder inequality, \dref{dd1.1fghyuisdakkkllljjjkk}, the continuity of the embedding $W^{1,2}(\Omega)\hookrightarrow L^6(\Omega)$, the Gagliardo--Nirenberg inequality and\dref{vgbhssddaqwswddaassffssff3.10deerfgghhjuuloollgghhhyhh}
 implies that there exists a positive constants $C_1,C_2$ and $C_3$ independent of $\varepsilon$ such that
\begin{equation}
\begin{array}{rl}
\disp\int_{\Omega}n_{\varepsilon}u_{\varepsilon}\cdot\nabla \phi\leq&\disp{\|\nabla \phi\|_{L^\infty(\Omega)}\|n_{\varepsilon}\|_{L^{\frac{6}{5}}(\Omega)}\|\nabla u_{\varepsilon}\|_{L^{2}(\Omega)}}\\
\leq&\disp{C_1\|n_{\varepsilon}\|_{L^{\frac{6}{5}}(\Omega)}\|\nabla u_{\varepsilon}\|_{L^{2}(\Omega)}}\\
\leq&\disp{C_2\|\nabla n_{\varepsilon}^{\frac{m}{2}}\|_{L^{2}(\Omega)}^{\frac{1}{3m-1}}\| n_{\varepsilon}^{\frac{m}{2}}\|_{L^{\frac{2}{m}}(\Omega)}^{\frac{2}{m}-\frac{1}{3m-1}}\|\nabla u_{\varepsilon}\|_{L^{2}(\Omega)}}\\
\leq&\disp{C_3(\|\nabla n_{\varepsilon}^{\frac{m}{2}}\|_{L^{2}(\Omega)}^{\frac{1}{3m-1}}+1)\|\nabla u_{\varepsilon}\|_{L^{2}(\Omega)}~~\mbox{for all}~~ t\in(0, T_{max}).}\\
\end{array}
\label{ddddfgcz2.5ghju48cfg924ghyuji}
\end{equation}
Next, with the help of the
Young inequality and $m>\frac{10}{9}$, inserting \dref{ddddfgcz2.5ghju48cfg924ghyuji} into \dref{ddddfgcz2.5ghju48cfg924ghyuffggji} and \dref{ghnjmk9161gyyhuug}, we derive that
\begin{equation}
\begin{array}{rl}
\disp\frac{1}{2}\int_{\Omega}{|u_{\varepsilon}|^2}+\frac{1}{2}\int_{\Omega}{|\nabla u_{\varepsilon}|^2}\leq &\disp{C_4(\|\nabla n_{\varepsilon}^{\frac{m}{2}}\|_{L^{2}(\Omega)}^{\frac{2}{3m-1}}+1)}\\
\leq &\disp{\frac{\delta_1}{2C_D}\|\nabla n_{\varepsilon}^{\frac{m}{2}}\|_{L^{2}(\Omega)}^{\frac{2}{3m-1}}+C_5}\\
\leq &\disp{\frac{\delta_1}{2}\int_{\Omega}\frac{D_{\varepsilon}(n_{\varepsilon})|\nabla n_{\varepsilon}|^2}{n_{\varepsilon}}+C_5~~\mbox{for all}~~ t\in(0, T_{max}),}\\
\end{array}
\label{ddddfgczxxccdd2.5ghju4cddfff8cfg924ghyuji}
\end{equation}
and some positive constants $C_4$ and $C_5.$  
\end{proof}

\begin{lemma}\label{ghjsgghhsdeedrfe116lemma70hhjj} Let $\frac{10}{9}< m\leq2$.
 There exist $\mu_0$ and  $C > 0$ independent of $\varepsilon$ such that for every  $\delta_i(i=2,3,4,5)>0$
 \begin{equation}
 \begin{array}{rl}
 &\disp{\frac{d}{dt}\int_{\Omega}\frac{|\nabla c_{\varepsilon}|^2}{c_{\varepsilon}}+\mu_0\int_{\Omega}c_{\varepsilon}|D^2\ln c_{\varepsilon}|^2+(\mu_0-\frac{\delta_2}{4}-\frac{\delta_3}{4})\int_{\Omega}\frac{|\nabla c_{\varepsilon}|^4}{c_{\varepsilon}^3}}\\
 \leq&\disp{ (\frac{\delta_4}{4}+\frac{\delta_5}{4})\int_{\Omega}\frac{D_{\varepsilon}(n_{\varepsilon})|\nabla n_{\varepsilon}|^2}{n_{\varepsilon}}+\frac{4}{\delta_2}\|c_0\|_{L^\infty(\Omega)}\int_{\Omega}|\nabla u_{\varepsilon}|^2+C~~~\mbox{for all}~~t\in (0, T_{max})}\\
\end{array}\label{vgbhsvvbbsddaqwswddaassffssff3.10deerfgghhjuuloollgghhhyhh}
\end{equation}
\end{lemma}
\begin{proof}
 Firstly, by calculation,  we derive that
\begin{equation}
 \begin{array}{rl}
 \disp\frac{d}{dt}\disp\int_{\Omega}\frac{|\nabla c_{\varepsilon}|^2}{c_{\varepsilon}}=&\disp{2\int_{\Omega}\frac{\nabla c_{\varepsilon}\cdot\nabla c_{\varepsilon t}}{c_{\varepsilon}}-\int_{\Omega}\frac{|\nabla c_{\varepsilon}|^2c_{\varepsilon t}}{c_{\varepsilon}^2}}\\
 =&\disp{-2\int_{\Omega}\frac{\Delta c_{\varepsilon} c_{\varepsilon t}}{c_{\varepsilon}}+\int_{\Omega}\frac{|\nabla c_{\varepsilon}|^2c_{\varepsilon t}}{c_{\varepsilon}^2}}\\
  =&\disp{-2\int_{\Omega}\frac{|\Delta c_{\varepsilon}|^2 }{c_{\varepsilon}}+2\int_{\Omega}\frac{\Delta c_{\varepsilon} n_{\varepsilon} c_{\varepsilon}}{c_{\varepsilon}}+2\int_{\Omega}\frac{\Delta c_{\varepsilon}}{c_{\varepsilon}}u_{\varepsilon}\cdot\nabla c_{\varepsilon}}\\
  &+\disp{\int_{\Omega}\frac{|\nabla c_{\varepsilon}|^2\Delta c_{\varepsilon} }{c_{\varepsilon}^2}-\int_{\Omega}\frac{|\nabla c_{\varepsilon}|^2n_{\varepsilon} c_{\varepsilon} }{c_{\varepsilon}^2}-\int_{\Omega}\frac{|\nabla c_{\varepsilon}|^2u_{\varepsilon}\cdot\nabla c_{\varepsilon} }{c_{\varepsilon}^2}~~~\mbox{for all}~~t\in(0, T_{max}).}\\
\end{array}\label{vgbhsvvbbsddacvvvvbbqwswddaassffssff3.10deerfgghhjuuloollgghhhyhh}
\end{equation}
 Due to (vi) of  Lemma \ref{lemmffffgga630jklhhjj} and the Young inequality, there exist $\mu_0>0$ and $C(\mu_0)>0$
 such that
\begin{equation}
 \begin{array}{rl}
  &\disp{-2\int_{\Omega}\frac{|\Delta c_{\varepsilon}|^2 }{c_{\varepsilon}}+\int_{\Omega}\frac{|\nabla c_{\varepsilon}|^2\Delta c_{\varepsilon} }{c_{\varepsilon}^2}\leq -\mu_0\int_{\Omega}c_{\varepsilon}|D^2\ln c_{\varepsilon}|^2-\mu_0\int_{\Omega}\frac{|\nabla c_{\varepsilon}|^4}{c_{\varepsilon}^3}+C(\mu_0)\int_{\Omega}c_{\varepsilon}}\\
\end{array}\label{vccvvgbhsvvbbsddacvvvvbbqwswddaassffssff3.10deerfgghhjuuloollgghhhyhh}
\end{equation}
for all $t\in (0, T_{max})$.
 On the other hand, for all $t\in (0, T_{max})$, with the help of the computing, the Young inequality and Lemma \ref{ghjssdeedrfe116lemma70hhjj} implies that for any $\delta_2 > 0 $
 \begin{equation}
 \begin{array}{rl}
\disp &2\int_{\Omega}\frac{\Delta c_{\varepsilon}}{c_{\varepsilon}}(u_{\varepsilon}\cdot\nabla c_{\varepsilon})-\int_{\Omega}\frac{|\nabla c_{\varepsilon}|^2}{c_{\varepsilon}^2}u_{\varepsilon}\cdot\nabla c_{\varepsilon}\\
= &\disp{2\int_{\Omega}\frac{|\nabla c_{\varepsilon}|^2}{c_{\varepsilon}^2}u_{\varepsilon}\cdot\nabla c_{\varepsilon}-2\int_{\Omega}\frac{1}{c_{\varepsilon}}\nabla c_{\varepsilon}\cdot(\nabla u_{\varepsilon}\nabla c_{\varepsilon})-2\int_{\Omega}\frac{1}{c_{\varepsilon}}u_{\varepsilon}\cdot D^2c_{\varepsilon}\nabla c_{\varepsilon}+2\int_{\Omega}\frac{1}{c_{\varepsilon}}u_{\varepsilon}\cdot D^2c_{\varepsilon}\nabla c_{\varepsilon}}\\
= &\disp{-2\int_{\Omega}\frac{1}{c_{\varepsilon}}\nabla c_{\varepsilon}\cdot(\nabla u_{\varepsilon}\nabla c_{\varepsilon})}\\
\leq &\disp{\frac{\delta_2}{4}\int_{\Omega}\frac{|\nabla c_{\varepsilon}|^4}{c_{\varepsilon}^3}+\frac{4}{\delta_2}\int_{\Omega}c_{\varepsilon}|\nabla u_{\varepsilon}|^2}\\
\leq &\disp{\frac{\delta_2}{4}\int_{\Omega}\frac{|\nabla c_{\varepsilon}|^4}{c_{\varepsilon}^3}+C_1\int_{\Omega}|\nabla u_{\varepsilon}|^2~~\mbox{for all}~~t\in (0, T_{max}),}\\
\end{array}\label{vgbhccvsvvbbsddacvvvvbbqwswddvvbbaassffssff3.10deerfgghhjuuloollgghhhyhh}
\end{equation}
where $C_1:=\frac{4}{\delta_2}\|c_0\|_{L^\infty(\Omega)}$. In view of integration by parts, the Young inequality, \dref{ghnjmk9161gyyhuug} and \dref{hnjmssddaqwswddaassffssff3.10deerfgghhjuuloollgghhhyhh}, we also derive that
\begin{equation}
 \begin{array}{rl}
 2\disp\int_{\Omega}\frac{\Delta c_{\varepsilon} n_{\varepsilon} c_{\varepsilon}}{c_{\varepsilon}}=&\disp{-2\int_{\Omega}\nabla n_{\varepsilon}\cdot \nabla c_{\varepsilon}}\\
 \leq&\disp{\frac{\delta_3}{4}\int_{\Omega}\frac{|\nabla c_{\varepsilon}|^4}{c_{\varepsilon}^3}+2^{\frac{4}{3}}\delta^{-\frac{1}{3}}_3\int_{\Omega}c_{\varepsilon}|\nabla n_{\varepsilon}|^{\frac{4}{3}}}\\
\leq&\disp{\frac{\delta_3}{4}\int_{\Omega}\frac{|\nabla c_{\varepsilon}|^4}{c_{\varepsilon}^3}+\frac{\delta_4}{4C_{D}}\int_{\Omega}n_{\varepsilon}^{m-2}|\nabla n_{\varepsilon}|^{2}+C_2\int_{\Omega}n_{\varepsilon}^{4-2m}c_{\varepsilon}^3}\\
\leq&\disp{\frac{\delta_3}{4}\int_{\Omega}\frac{|\nabla c_{\varepsilon}|^4}{c_{\varepsilon}^3}+\frac{\delta_4}{4}\int_{\Omega}\frac{D_{\varepsilon}(n_{\varepsilon})|\nabla n_{\varepsilon}|^2}{n_{\varepsilon}}+C_3\int_{\Omega}n_{\varepsilon}^{4-2m}~~
\mbox{for all}~~t\in (0, T_{max}),}\\
\end{array}\label{vgbhsvvbbsddacvvvvcvvvbbqwswddaassffssff3.10deerfgghhjuuloollgghhhyhh}
\end{equation}
where $\delta_3,\delta_4,C_2:=C_2(\delta_3,\delta_4), C_3:=C_3(\delta_3,\delta_4,\|c_0\|_{L^\infty(\Omega)})$  are positive constants.

Case $\frac{10}{9}<m<\frac{3}{2}$:
Due to the Gagliardo--Nirenberg inequality and \dref{vgbhssddaqwswddaassffssff3.10deerfgghhjuuloollgghhhyhh}, we conclude that
\begin{equation}
\begin{array}{rl}
C_3\disp\int_{\Omega}n_{\varepsilon}^{4-2m}=&\disp{C_3\| n_{\varepsilon}^{\frac{m}{2}}\|_{L^{\frac{2(4-2m)}{m}}(\Omega)}^{\frac{2(4-2m)}{m}}}\\
\leq &\disp{C_4\|\nabla n_{\varepsilon}^{\frac{m}{2}}\|_{L^{2}(\Omega)}^{\frac{2(4-2m)\mu_1}{m}}\| n_{\varepsilon}^{\frac{m}{2}}\|_{L^{\frac{2}{m}}(\Omega)}^{\frac{2(4-2m)(1-\mu_1)}{m}}+\| n_{\varepsilon}^{\frac{m}{2}}\|_{L^{\frac{2}{m}}(\Omega)}^{\frac{2(4-2m)}{m}}}\\
\leq &\disp{C_5(\|\nabla n_{\varepsilon}^{\frac{m}{2}}\|_{L^{2}(\Omega)}^{\frac{2(4-2m)\mu_1}{m}}+1)}\\
= &\disp{C_5(\|\nabla n_{\varepsilon}^{\frac{m}{2}}\|_{L^{2}(\Omega)}^{\frac{6(3-2m)}{3m-1}}+1)~~\mbox{for all}~~t\in (0, T_{max}),}\\
\end{array}\label{vgbhsvvbbsddffffggdaccvvcvvvvcvvvbbqwswddaassffssff3.10deerfgghhjuuloollgghhhyhh}
\end{equation}
where $C_4$ and $C_5$ are positive constants,
$$\mu_1=\frac{\frac{3m}{2}-\frac{3m}{2(4-2m)}}{\frac{3m-1}{2}}\in(0,1).$$
Now, in view of $m>\frac{10}{9},$ with the help of the Young inequality and \dref{vgbhsvvbbsddffffggdaccvvcvvvvcvvvbbqwswddaassffssff3.10deerfgghhjuuloollgghhhyhh}, for any $\delta_5>0$, we have
\begin{equation}
\begin{array}{rl}
&\disp{C_3\int_{\Omega}n_{\varepsilon}^{4-2m}\leq\frac{\delta_5}{4C_D}\|\nabla n_{\varepsilon}^{\frac{m}{2}}\|_{L^{2}(\Omega)}^{2}+C_6~~\mbox{for all}~~t\in (0, T_{max})}\\
\end{array}\label{vgbhccvvvvsvvbbsddffffggdaccvvcvvvvcvvvbbqwswddaassffssff3.10deerfgghhjuuloollgghhhyhh}
\end{equation}
with some $C_6>0$.

Case $\frac{3}{2}\leq m\leq2$: With the help of the Young inequality and \dref{vgbhssddaqwswddaassffssff3.10deerfgghhjuuloollgghhhyhh}, we derive that
\begin{equation}
\begin{array}{rl}
C_3\disp\int_{\Omega}n_{\varepsilon}^{4-2m}
\leq &\disp{\int_{\Omega}n_{\varepsilon}+C_7}\\
\leq &\disp{C_8~~\mbox{for all}~~t\in (0, T_{max})}\\
\end{array}\label{vgbhsvvbbsdffssff3.10deerfgghhjuuloollgghhhyhh}
\end{equation}
where $C_7$ and $C_8$ are positive constants independent of $\varepsilon.$
Finally, collecting
 \dref{vccvvgbhsvvbbsddacvvvvbbqwswddaassffssff3.10deerfgghhjuuloollgghhhyhh}--\dref{vgbhsvvbbsdffssff3.10deerfgghhjuuloollgghhhyhh}
 and \dref{vgbhsvvbbsddacvvvvbbqwswddaassffssff3.10deerfgghhjuuloollgghhhyhh}, we can get the results.
\end{proof}

\begin{lemma}\label{ghjhhjssjkkllsgghhsdeedrfe116lemma70hhjj}  Let $\frac{10}{9}< m\leq2$ and $\delta>0$.
There is $C > 0$ such that for any 
$\delta_6$ and  $\delta_7$
 \begin{equation}
 \begin{array}{rl}
 \disp\frac{d}{dt}\int_{\Omega}n_{\varepsilon} \ln n_{\varepsilon}+
(1-\frac{\delta_7}{4})\int_{\Omega}\frac{D_{\varepsilon}(n_{\varepsilon})|\nabla n_{\varepsilon}|^2}{n_{\varepsilon}}\leq \frac{\delta_6}{4}\int_{\Omega}\frac{|\nabla c_{\varepsilon}|^4}{c_{\varepsilon}^3} +C~\mbox{for all}~t\in (0, T_{max}).
\end{array}\label{vgccvssbbbbhsvvbbsddaqwswddaassffssff3.10deerfgghhjuuloollgghhhyhh}
\end{equation}
\end{lemma}
\begin{proof}
Using these estimates and the first equation of \dref{1.1fghyuisda}, from integration by parts we obtain
\begin{equation}
 \begin{array}{rl}
\frac{d}{dt}\disp\int_{\Omega}n_{\varepsilon} \ln n_{\varepsilon} =&\disp{\int_{\Omega}n_{\varepsilon t} \ln n_{\varepsilon}+
\int_{\Omega}n_{\varepsilon t}}\\
=&\disp{\int_{\Omega}\nabla\cdot(D_{\varepsilon}(n_{\varepsilon})\nabla n_{\varepsilon} ) \ln n_{\varepsilon}-
\int_{\Omega}\ln n_{\varepsilon}\nabla\cdot(n_{\varepsilon}S_\varepsilon(x, n_{\varepsilon}, c_{\varepsilon})\nabla c_{\varepsilon})-\int_{\Omega}\ln n_{\varepsilon}u_{\varepsilon}\cdot\nabla n_{\varepsilon}}\\
\leq&\disp{-\int_{\Omega}\frac{D_{\varepsilon}(n_{\varepsilon})|\nabla n_{\varepsilon}|^2}{n_{\varepsilon}}+
\int_{\Omega} S_0(c_{\varepsilon})|\nabla n_{\varepsilon}||\nabla c_{\varepsilon}|}\\
\end{array}\label{vgccvsckkcvvsbbbbhsvvbbsddaqwswddaassffssff3.10deerfgghhjuuloollgghhhyhh}
\end{equation}
for all $t\in (0, T_{max})$. 
Now, in view of \dref{hnjmssddaqwswddaassffssff3.10deerfgghhjuuloollgghhhyhh}, employing the same argument of \dref{vgbhsvvbbsddacvvvvcvvvbbqwswddaassffssff3.10deerfgghhjuuloollgghhhyhh}--\dref{vgbhsvvbbsdffssff3.10deerfgghhjuuloollgghhhyhh}, for any $\delta_6>0$ and $\delta_7>0$, we conclude that
\begin{equation}
 \begin{array}{rl}
 &\disp{\int_{\Omega} S_0(c_{\varepsilon})|\nabla n_{\varepsilon}||\nabla c_{\varepsilon}|}\\
\leq &\disp{S_0(\|c_0\|_{L^\infty(\Omega)})\int_{\Omega}|\nabla n_{\varepsilon}| |\nabla c_{\varepsilon}|}\\
 \leq&\disp{\frac{\delta_6}{4}\int_{\Omega}\frac{|\nabla c_{\varepsilon}|^4}{c_{\varepsilon}^3}+\frac{\delta_7}{4}\int_{\Omega}\frac{D_{\varepsilon}(n_{\varepsilon})|\nabla n_{\varepsilon}|^2}{n_{\varepsilon}}+C_9~~
\mbox{for all}~~t\in (0, T_{max}).}\\
\end{array}\label{vgbhsvvbbsddacvvvccvcvvvbbqwswddaassffssff3.10deerfgghhjuuloollgghhhyhh}
\end{equation}
with $C_9>0$ independent of $\varepsilon.$
Now, in conjunction with \dref{vgccvsckkcvvsbbbbhsvvbbsddaqwswddaassffssff3.10deerfgghhjuuloollgghhhyhh} and \dref{vgbhsvvbbsddacvvvccvcvvvbbqwswddaassffssff3.10deerfgghhjuuloollgghhhyhh}, we get the results.
This completes the proof of Lemma \ref{ghjhhjssjkkllsgghhsdeedrfe116lemma70hhjj}.
\end{proof}
Properly combining Lemmata \ref{lemma630jklhhjj}--\ref{ghjhhjssjkkllsgghhsdeedrfe116lemma70hhjj}, we arrive at the following Lemma, which plays a key rule in obtaining
 the existence of solutions to \dref{1.1fghyuisda}.
\begin{lemma}\label{lemmakkllgg4563025xxhjklojjkkk}
Let $\frac{10}{9}< m\leq2$ and  $S$ satisfy  \dref{x1.73142vghf48rtgyhu}--\dref{x1.73142vghf48gg}.
Suppose that \dref{ghnjmk9161gyyhuug} and \dref{dd1.1fghyuisdakkkllljjjkk}--\dref{ccvvx1.731426677gg}
holds.
Then there exists $C>0$ such that the solution of \dref{1.1fghyuisda} satisfies
\begin{equation}
\begin{array}{rl}
&\disp{\int_{\Omega}n_{\varepsilon}(\cdot,t)\ln n_{\varepsilon}(\cdot,t)+\int_{\Omega}|\nabla\sqrt{c_{\varepsilon}}(\cdot,t)|^2+\int_{\Omega}|u_{\varepsilon}(\cdot,t)|^2\leq C}\\
\end{array}
\label{czfvgb2.5ghhjuyuiihjj}
\end{equation}
for all $t\in(0, T_{max})$.
Moreover,
for each $T\in(0, T_{max})$, one can find a constant $C > 0$ such that
\begin{equation}
\begin{array}{rl}
&\disp{\int_{0}^T\int_{\Omega}  n_{\varepsilon}^{m-2} |\nabla {n_{\varepsilon}}|^2\leq C}\\
\end{array}
\label{bnmbncz2.5ghhjuyuiihjj}
\end{equation}
and
\begin{equation}
\begin{array}{rl}
&\disp{\int_{0}^T\int_{\Omega} |\nabla {c_{\varepsilon}}|^4\leq C}\\
\end{array}
\label{vvcz2.5ghhjuyuiihjj}
\end{equation}
as well as
\begin{equation}
\begin{array}{rl}
&\disp{\int_{0}^T\int_{\Omega}c_{\varepsilon}|D^2\ln c_{\varepsilon}|^2\leq C.}\\
\end{array}
\label{cvffvbgvvcz2.5ghhjuyuiihjj}
\end{equation}
\end{lemma}
\begin{proof}
Take an evident linear combination of the inequalities provided by Lemmata \ref{lemma630jklhhjj}--\ref{ghjhhjssjkkllsgghhsdeedrfe116lemma70hhjj}, we conclude that there exists a positive constant $C_1>0$ such that
\begin{equation}
 \begin{array}{rl}
 &\disp{\frac{d}{dt}\left(\int_{\Omega}\frac{|\nabla c_{\varepsilon}|^2}{c_{\varepsilon}}+L\int_{\Omega}n_{\varepsilon} \ln n_{\varepsilon}+K|u_{\varepsilon}|^2\right)+(K-\frac{4}{\delta_2}\|c_0\|_{L^\infty(\Omega)})\int_{\Omega}|\nabla u_{\varepsilon}|^2+\mu_0\int_{\Omega}c_{\varepsilon}|D^2\ln c_{\varepsilon}|^2}\\
 &\disp{+[(\mu_0-\frac{\delta_2}{4}-\frac{\delta_3}{4})-L\frac{\delta_6}{4}]\int_{\Omega}\frac{|\nabla c_{\varepsilon}|^4}{c_{\varepsilon}^3}
 +[L(1-\frac{\delta_7}{4})-\frac{\delta_4}{4}-\frac{\delta_5}{4}-K\delta_1]\int_{\Omega}\frac{D_{\varepsilon}(n_{\varepsilon})|\nabla n_{\varepsilon}|^2}{n_{\varepsilon}}}\\
 \leq&\disp{ C_1~~~\mbox{for all}~~t\in (0, T_{max,\varepsilon}),}\\
\end{array}\label{vgbccvbbffeerfgghhjuuloollgghhhyhh}
\end{equation}
where $K,L$ are positive constants. Now, choosing $\delta_7=1,$ $\delta_6=\frac{\mu_0}{L},\delta_3=\mu_0,\delta_4=\delta_5=L,\delta_1=\frac{L}{8K}$ and
$\delta_2=\frac{8}{K}\|c_0\|_{L^\infty(\Omega)}$  and $K$ large enough such that $\frac{8}{K}\|c_0\|_{L^\infty(\Omega)}<\mu_0$
in \dref{vgbccvbbffeerfgghhjuuloollgghhhyhh}, we derive that
 \begin{equation}
 \begin{array}{rl}
 &\disp{\frac{d}{dt}\left(\int_{\Omega}\frac{|\nabla c_{\varepsilon}|^2}{c_{\varepsilon}}+L\int_{\Omega}n_{\varepsilon} \ln n_{\varepsilon}+K|u_{\varepsilon}|^2\right)+\frac{K}{2}\int_{\Omega}|\nabla u_{\varepsilon}|^2+\mu_0\int_{\Omega}c_{\varepsilon}|D^2\ln c_{\varepsilon}|^2}\\
 &\disp{+\frac{\mu_0}{4}\int_{\Omega}\frac{|\nabla c_{\varepsilon}|^4}{c_{\varepsilon}^3}
 +\frac{L}{8}\int_{\Omega}\frac{D_{\varepsilon}(n_{\varepsilon})|\nabla n_{\varepsilon}|^2}{n_{\varepsilon}}}\\
 \leq&\disp{ C_2~~~\mbox{for all}~~t\in (0, T_{max,\varepsilon}).}\\
\end{array}\label{vgbccvbbffeerfgghcchjuuloollgghhhyhh}
\end{equation}
and some positive constant $C_2.$
Hence, by some basic calculation,
we can conclude  \dref{czfvgb2.5ghhjuyuiihjj}--\dref{cvffvbgvvcz2.5ghhjuyuiihjj}.
\end{proof}
\begin{remark}\label{ffgg44remark}From Lemma \ref{lemmakkllgg4563025xxhjklojjkkk},  $m>\frac{10}{9}$ yields the \dref{czfvgb2.5ghhjuyuiihjj},
which is the natural energy functional of   $3D$ chemotaxis-stokes system with  nonlinear diffusion and rotation. Furthermore, Lemma \ref{lemmakkllgg4563025xxhjklojjkkk} expands the contents  of \cite{Winkler11215} (see also the  introduction of \cite{Cao22119}, Wang et. al \cite{Wang11215,Wang21215}).
Moreover, rely on  Lemma \ref{lemmakkllgg4563025xxhjklojjkkk},  we also conclude that the large diffusion exponent $m$ ($>\frac{10}{9}$) benefits the existence  of
solutions to \dref{1.1fghyuisda}.
\end{remark}

Employing almost exactly the same arguments as in the proof of Lemma 3.5  and Lemma 3.6 in \cite{Tao71215} (the minor necessary changes are left as an easy exercise to the reader), we conclude the following Lemmata:
\begin{lemma}\label{lemmaghjmk4563025xxhjklojjkkk}
Let $p>1$.
Then the solution of \dref{1.1fghyuisda} from Lemma \ref{lemma70} satisfies
\begin{equation}
\begin{array}{rl}
&\disp{\frac{1}{{p}}\frac{d}{dt}\|n_{\varepsilon}\|^{{{p}}}_{L^{{p}}(\Omega)}+
\frac{C_D(p-1)}{2}\int_{\Omega}n_{\varepsilon}^{m+p-3} |\nabla n_{\varepsilon}|^2 \leq C\int_\Omega n_{\varepsilon}^{p+1-m}|\nabla c_{\varepsilon}|^2   ,}\\
\end{array}
\label{cz2.5ghhjuyuiihjj}
\end{equation}
where $C>0$ is a positive constant depends on $p$ and $C_D$.
\end{lemma}

\begin{lemma}\label{lemmaghjmk45ss63025xxhjklojjkkk}
Let  $q>1.$
Then the solution of \dref{1.1fghyuisda} from Lemma \ref{lemma70} satisfies
\begin{equation}
\begin{array}{rl}
&\disp{\frac{1}{{2{q}}}\frac{d}{dt}\|\nabla c_{\varepsilon}{}\|^{{{2{q}}}}_{L^{{2{q}}}(\Omega)}+\frac{({q}-1)}{{{q}^2}}\int_{\Omega}\left|\nabla |\nabla c_{\varepsilon}{}|^{{q}}\right|^2+\frac{1}{2}\int_\Omega  |\nabla c_{\varepsilon}{}|^{2{q}-2}|D^2c_{\varepsilon}{}|^2}\\
\leq&\disp{C(\int_\Omega n_{\varepsilon}^2{} |\nabla c_{\varepsilon}{}|^{2{q}-2}
+\int_\Omega |Du_{\varepsilon}{}| |\nabla c_{\varepsilon}{}|^{2{q}})   ,}\\
\end{array}
\label{cz2.5ghhjssuyuiihjj}
\end{equation}
where $C>0$ is a positive constant depends on $q $ and $\|c_0\|_{L^\infty(\Omega)}$.
\end{lemma}
%

In the following, we will estimate the integrals on the right-hand sides of \dref{cz2.5ghhjuyuiihjj} and \dref{cz2.5ghhjssuyuiihjj}.
%
%
%
%
%
%
Indeed,
%
we first give  the  following Lemma which plays an important rule in estimating the boundedness of  $|Du_{\varepsilon}{}| |\nabla c_{\varepsilon}{}|^{2{q}}$.
\begin{lemma}(Lemma 3.7 of \cite{Winkler11215})\label{ffglffhhemmaccvbbfggg78630jklhhjj}
 For any $\varepsilon\in (0,1),$ we have
\begin{equation}
 \begin{array}{rl}
 \disp{\frac{d}{dt}\int_{\Omega}|A^{\frac{1}{2}}u_{\varepsilon}|^2+\int_{\Omega}|Au_{\varepsilon}|^2\leq\|\nabla\phi\|_{L^\infty(\Omega)}^2\int_\Omega n^2_{\varepsilon}~~~\mbox{for all}~~t\in (0, T_{max}).}
\end{array}\label{hnjmvvbbbssddaqwswffgddaassffssff3.10deerfgghhjuuloollgghhhyhh}
\end{equation}
\end{lemma}

Now, in order to estimate $\int_\Omega n_{\varepsilon}^2{} |\nabla c_{\varepsilon}{}|^{2{q}-2},$  we
recall the  following Lemma which comes from  Lemma 2.7 of \cite{Tao71215} (see also \cite{Zhengddffsdsd6})
%
%
%
\begin{lemma}(Lemma 2.7 of \cite{Tao71215})\label{lemmafcvg4563025xxhjklojjkkkgyhuihhu}
Assume
that
%
the initial data $(n_0,c_0,u_0)$ fulfills \dref{ccvvx1.731426677gg}. Let
$
p_0 \in(0, 9(m-1)),m>1
$
and $T_0 > 0$, then one can find a constant $C(p_0, T) > 0$ such that
\begin{equation}
\begin{array}{rl}
&\disp{\|{n}_\varepsilon(\cdot,t)\|_{L^{p_0}(\Omega)}\leq C(p_0, T)~~~\mbox{for all}~~ t\in (0, T)  }\\
\end{array}
\label{cz2.5ffvvbbghhjuyuiihjjrftg}
\end{equation}
holds with \begin{equation}
T :=\min\{T_0,T_{max}\}.
\label{hjnmkcz2.5ffvvbbghhjuyuiihjjrftg}
\end{equation}
\end{lemma}
\begin{lemma}\label{lemma4563025xxhjklojjkkkgyhuihhu}
Assuming that
$p_0\in [1,9(m-1)),m>1 $
and
${q}>1$.
If
\begin{equation}
\max\{1,\frac{p_0}{3}+1-m,m-1+p_0\frac{ q}{ q+1 }\}< p <
[2(m-1)+\frac{2p_0}{3}]q+m-1,
\label{derfrscz2.5ghhjuyuiihjjrftg}
\end{equation}
then for all
small $\delta>0$,
we can find a constant $C:= C(p,{q},\delta) > 0$ such that
\begin{equation}
\begin{array}{rl}
&\disp{\int_\Omega n^{p+1-m}_{\varepsilon} |\nabla {c}_{\varepsilon}|^2 \leq\delta\int_{\Omega} |n_{\varepsilon}^{\frac{p+m-1}{2}}|^2+\delta\||\nabla{c}_{\varepsilon}|^{{q}-1}D^2{c}_{\varepsilon}\|_{L^2(\Omega)}^2+C  }\\
\end{array}
\label{cz2.5ghhjuyuiihddfffjjrftg}
\end{equation}
for all $t\in(0,T)$, where $T$ is given by \dref{hjnmkcz2.5ffvvbbghhjuyuiihjjrftg}.
\end{lemma}
\begin{proof}
We apply the H\"{o}lder inequality with exponents $q+1$
and $\frac{q+1}{q}$ to obtain
\begin{equation}
\begin{array}{rl}
J_1&:=\disp{ \int_\Omega n^{p+1-m}_{\varepsilon} |\nabla {c}_{\varepsilon}|^2}\\
&\leq\disp{ \left(\int_\Omega n^{ \frac{q+1}{q} (p+1-m)}_{\varepsilon}\right)^{\frac{1}{ \frac{q+1}{q} }}\left(\int_\Omega |\nabla {c}_{\varepsilon}|^{2(q+1)}\right)^{\frac{1}{\kappa}}}\\
&=\disp{ \|   n^{\frac{m+p-1}{2}}_{\varepsilon}\|^{\frac{2(p+1-m)}{m+p-1}}_{L^{\frac{2 \frac{q+1}{q} (p+1-m)}{m+p-1}}(\Omega)}
 \|\nabla {c}_{\varepsilon}\|_{L^{2(q+1)}(\Omega)}^2.}\\
\end{array}
\label{cz2.57151hhkkhhhjukildrfthjjhhhhh}
\end{equation}
%
Due to $p>m-1+p_0\frac{q}{q+1},m>1 $ and $q>1$, 
we have
$$\frac{p_0}{m+p-1}\leq\frac{ \frac{q+1}{q} (p+1-m) }{m+p-1}\leq3,$$
which together with Lemma \ref{lemmafcvg4563025xxhjklojjkkkgyhuihhu} and the Gagliardo--Nirenberg inequality (see e.g. \cite{Zheng00}) implies that there exist positive constants $C_1$ and $C_2$ such that
\begin{equation}
\begin{array}{rl}
&\disp{\|   n^{\frac{m+p-1}{2}}_{\varepsilon}\|
^{\frac{2(p+1-m)}{m+p-1}}_{L^{\frac{2 \frac{q+1}{q} (p+1-m)}{m+p-1}}(\Omega)}}
\\
\leq&\disp{C_1(\|n_{\varepsilon}^{\frac{p+m-1}{2}}\|_{L^2(\Omega)}^{\mu_1}\|   n^{\frac{m+p-1}{2}}_{\varepsilon}\|_{L^\frac{2p_0}{m+p-1}(\Omega)}^{1-\mu_1}+\|  n^{\frac{m+p-1}{2}}_{\varepsilon}\|_{L^\frac{2p_0}{m+p-1}(\Omega)})^{\frac{2(p+1-m)}{m+p-1}}}\\
\leq&\disp{C_{2}(\|n_{\varepsilon}^{\frac{p+m-1}{2}}\|_{L^2(\Omega)}^{\frac{2(p+1-m)\mu_1}{m+p-1}}+1)}\\
=&\disp{C_{2}(\|n_{\varepsilon}^{\frac{p+m-1}{2}}\|_{L^2(\Omega)}^{\frac{3(\frac{p+1-m}{p_0}-\frac{q}{q+1})}{-\frac{1}{2}+\frac{3({m+p-1})}{2p_0}}}+1),}\\
\end{array}
\label{cz2.563022222ikopl2sdfg44}
\end{equation}
where 
$$\mu_1=\frac{\frac{3({m+p-1})}{2p_0}-\frac{3({m+p-1})}{2 \frac{q+1}{q} (p+1-m)}}{-\frac{1}{2}+\frac{3({m+p-1})}{2p_0}}=
({m+p-1})\frac{\frac{3}{2p_0}-\frac{3}{2 \frac{q+1}{q} (p+1-m)}}{-\frac{1}{2}+\frac{3({m+p-1})}{2p_0}}\in(0,1).$$
On the other hand, with the help of Lemma \ref{ghjssdeedrfe116lemma70hhjj} and Lemma \ref{drfe116lemma70hhjj}, we conclude that 
there exist some positive constants $C_3$ and $C_{4}$ such that
\begin{equation}
\begin{array}{rl}
 \|\nabla {c}_{\varepsilon}\|_{L^{2(q+1)}(\Omega)}^2\leq&\disp{C_3\||\nabla{c}_{\varepsilon}|^{q-1}D^2{c}_{\varepsilon}\|_{L^2(\Omega)}^{\frac{2}{q+1}}
 \|{c}_{\varepsilon}\|_{L^\infty(\Omega)}^{\frac{2(q-1)}{(2q-1)(q+1)}}+C_3\|{c}_{\varepsilon}\|_{L^\infty(\Omega)}^2}
\\
\leq&\disp{C_4(\||\nabla{c}_{\varepsilon}|^{q-1}D^2{c}_{\varepsilon}\|_{L^2(\Omega)}^{\frac{2}{q+1}}+1).}\\
\end{array}
\label{cz2.563022222ikopl255}
\end{equation}

Inserting \dref{cz2.563022222ikopl2sdfg44}--\dref{cz2.563022222ikopl255} into \dref{cz2.57151hhkkhhhjukildrfthjjhhhhh} and using the Young inequality and \dref{derfrscz2.5ghhjuyuiihjjrftg}, we derive that there exist positive constants $C_5$ and $C_6$ such that for every $\delta>0,$
\begin{equation}
\begin{array}{rl}
J_1=&\disp{C_{5}(\|n_{\varepsilon}^{\frac{p+m-1}{2}}\|_{L^2(\Omega)}^{\frac{3[{\frac{p+1-m}{p_0} -\frac{1}{ \frac{q+1}{q} }}]}{-\frac{1}{2}+\frac{3(m+{p}-1)}{2p_0}}}+1)
(\||\nabla{c}_{\varepsilon}|^{q-1}D^2{c}_{\varepsilon}\|_{L^2(\Omega)}^{\frac{2}{q+1}}+1)}\\
\leq&\disp{\delta\int_{\Omega} |n_{\varepsilon}^{\frac{p+m-1}{2}}|^2+\delta\||\nabla{c}_{\varepsilon}|^{q-1}D^2{c}_{\varepsilon}\|_{L^2(\Omega)}^2
+C_6~~\mbox{for all}~~ t\in(0,T).  }\\
\end{array}
\label{cz2.57151hhkkhhhjukildrftdfrtgyhu}
\end{equation}
Here we have use the fact that
$$0<{\frac{3[{\frac{p+1-m}{p_0} -\frac{1}{ \frac{q+1}{q} }}]}{-\frac{1}{2}+\frac{3(m+{p}-1)}{2p_0}}}+{\frac{2}{q+1}}<2$$
and
$${\frac{3[{\frac{p+1-m}{p_0} -\frac{1}{ \frac{q+1}{q} }}]}{-\frac{1}{2}\frac{3(m+{p}-1)}{2p_0}}}>0,{\frac{2}{q+1}}>0.$$
\end{proof}

\begin{lemma}\label{lemma4563025xxhjklojjkkkgyhuiyuiko}
Assuming that
${q}>\max\{1,p_0-1\}$
If
\begin{equation}
p > \max\{1,1-m+\frac{q+1}{3},
q+2-\frac{2p_0}{3}-m\},\label{ghnjjcz2.5ghhjuyuiihjjrftg}
\end{equation}
then for all
small $\delta>0$,
we can find a constant $C:= C(p,q,p_0,\delta) > 0$ such that
\begin{equation}
\begin{array}{rl}
&\disp{\int_\Omega n^2_{\varepsilon} |\nabla {c}_{\varepsilon}|^{2{q}-2}+\int_\Omega n^2_{\varepsilon} \leq\delta\int_{\Omega} |n_{\varepsilon}^{\frac{p+m-1}{2}}|^2+\delta\||\nabla{c}_{\varepsilon}|^{q-1}D^2{c}_{\varepsilon}\|_{L^2(\Omega)}^2+C  }\\
\end{array}
\label{cz2.5ghhjuyuiihjjhhhhrftg}
\end{equation}
for all $t\in(0,T)$, where $T$ is given by \dref{hjnmkcz2.5ffvvbbghhjuyuiihjjrftg}. 
\end{lemma}

\begin{proof}
Firstly,  in light  of H\"{o}lder inequality with exponents $\frac{q+1}{{q}-1}$
and  $\frac{q+1}{2}$, we obtain
%
\begin{equation}
\begin{array}{rl}
J_2&:=\disp{ \int_\Omega  n^2_{\varepsilon} |\nabla {c}_{\varepsilon}|^{2{q}-2}}\\
&\leq\disp{ \left(\int_\Omega n^{q+1}_{\varepsilon}\right)^{\frac{2}{q+1}}\left(\int_\Omega |\nabla {c}_{\varepsilon}|^{2(q+1)}\right)^{\frac{1}{\frac{q+1}{q-1}}}}\\
&=\disp{ \|   n^{\frac{m+p-1}{2}}_{\varepsilon}\|^{\frac{4}{m+p-1}}_{L^{\frac{2(q+1)}{m+p-1}}(\Omega)}
 \|\nabla {c}_{\varepsilon}\|_{L^{2(q+1)}(\Omega)}^{(2{q}-2)}.}\\
\end{array}
\label{cz2.57151hhkkhhhjukildrfttyuijkoghyu66}
\end{equation}

On the other hand, in view of  $p>1-m+\frac{q+1}{3}$ and Lemma \ref{lemmafcvg4563025xxhjklojjkkkgyhuihhu} and  the Gagliardo--Nirenberg inequality we conclude that there exist positive constants $C_{1}$ and $C_{2}$ such that
\begin{equation}
\begin{array}{rl}
&\disp{\|   n^{\frac{m+p-1}{2}}_{\varepsilon}\|^{\frac{4}{m+p-1}}_{L^{\frac{2(q+1)}{m+p-1}}(\Omega)}}
\\
\leq&\disp{C_{1}(\|\nabla n_{\varepsilon}^{\frac{p+m-1}{2}}\|_{L^2(\Omega)}^{\mu_3}\|   n^{\frac{m+p-1}{2}}_{\varepsilon}\|_{L^\frac{2p_0}{m+p-1}(\Omega)}^{(1-\mu_3)}+\|   n^{\frac{m+p-1}{2}}_{\varepsilon}\|_{L^\frac{2p_0}{m+p-1}(\Omega)})^{\frac{4}{m+p-1}}}\\
\leq&\disp{C_{2}(\|\nabla n_{\varepsilon}^{\frac{p+m-1}{2}}\|_{L^2(\Omega)}^{\frac{4\mu_3}{m+p-1}}+1).}\\
\end{array}
\label{cz2.563022222ikopl2sdfgggjjkkk66}
\end{equation}
Here
$$\mu_3=\frac{\frac{3{(m+p-1)}}{2p_0}-\frac{3{(m+p-1)}}{2(q+1)}}{-\frac{1}{2}+\frac{3{(m+p-1)}}{2p_0}}=
{(m+p-1)}\frac{\frac{3}{2p_0}-\frac{3}{2(q+1)}}{-\frac{1}{2}+\frac{3{(m+p-1)}}{2p_0}}\in(0,1).$$
On the other hand,  according to
%
%
 Lemma \ref{ghjssdeedrfe116lemma70hhjj} and the Gagliardo--Nirenberg inequality, we can find
some positive constants $C_{3}$ and $C_{4}$ such that
\begin{equation}
\begin{array}{rl}
 \|\nabla {c}_{\varepsilon}\|_{L^{2(q+1)}(\Omega)}^{(2{q}-2)}\leq&\disp{C_{3}\||\nabla{c}_{\varepsilon}|^{q-1}D^2{c}_{\varepsilon}\|_{L^2(\Omega)}
 ^{{\frac{2(q-1)}{q+1}}}
 \|{c}_{\varepsilon}\|_{L^\infty(\Omega)}^{\frac{2(q-1)}{q+1}}+C_{3}\|{c}_{\varepsilon}\|_{L^\infty(\Omega)}^{2q-2}}
\\
\leq&\disp{C_{4}(\||\nabla{c}_{\varepsilon}|^{q-1}D^2{c}_{\varepsilon}\|_{L^2(\Omega)}^{{\frac{2(q-1)}{q+1}}}+1).}\\
\end{array}
\label{czsdefr2.563022222ikopl255}
\end{equation}
Now, observing that
$$0<{\frac{3(\frac{2}{p_0}-\frac{2}{q+1})}{-\frac{1}{2}+\frac{3{(m+p-1)}}{2p_0}}}+{{\frac{2(q-1)}{q+1}}}<2$$
and
$${\frac{3(\frac{2}{p_0}-\frac{2}{q+1})}{-\frac{1}{2}+\frac{3{(m+p-1)}}{2p_0}}}>0,{{\frac{2(q-1)}{q+1}}}>0,$$
hence, inserting \dref{cz2.563022222ikopl2sdfgggjjkkk66}--\dref{czsdefr2.563022222ikopl255} into \dref{cz2.57151hhkkhhhjukildrfttyuijkoghyu66} and using \dref{ghnjjcz2.5ghhjuyuiihjjrftg} and the Gagliardo--Nirenberg inequality, we derive that for any $\delta>0,$
\begin{equation}
\begin{array}{rl}
J_2\leq&\disp{C_{5}(\|\nabla n_{\varepsilon}^{\frac{p+m-1}{2}}\|_{L^2(\Omega)}^{
\frac{3(\frac{2}{p_0}-\frac{2}{q+1})}{-\frac{1}{2}+\frac{3{(m+p-1)}}{2p_0}}}+1)
(\||\nabla{c}_{\varepsilon}|^{q-1}D^2{c}_{\varepsilon}\|_{L^2(\Omega)}^{{\frac{2(q-1)}{q+1}}}+1)}\\
\leq&\disp{\delta\int_{\Omega} |\nabla n_{\varepsilon}^{\frac{p+m-1}{2}}|^2+\delta\||\nabla{c}_{\varepsilon}|^{q-1}D^2{c}_{\varepsilon}\|_{L^2(\Omega)}^2
+C_{6}~~\mbox{for all}~~ t\in(0,T).  }\\
\end{array}
\label{cz2.57151hhkkhhhjukildrftdfrtgyhugg}
\end{equation}
with some positive  constants $C_5$ and $C_6.$

Since the integral $\int_\Omega n^2_{\varepsilon}$ can be estimated similarly upon a straightforward simplification of the above
argument, this establishes \dref{cz2.5ghhjuyuiihjjhhhhrftg}.
\end{proof}

\begin{lemma}\label{lemma4563025xxhjklojjkkkgyhuirrtt}
For any $1\leq l<\frac{3}{2}$,
if
$$1<q < \frac{2l+3}{3},$$
then for all
small $\delta>0$, the solution of \dref{1.1fghyuisda} from Lemma \ref{lemma70} satisfies
\begin{equation}
\begin{array}{rl}
&\disp{\int_\Omega |Du{}_{\varepsilon}||\nabla c_{\varepsilon}{}|^{2{q}}\leq\delta\int_{\Omega} |Au_{\varepsilon}|^2+
\delta\|\nabla |\nabla c_{\varepsilon}|^{q}\|_{L^2(\Omega)}^2+C ~~ \mbox{for all}~~ t\in(0, T), }\\
\end{array}
\label{cz2.5ghhjuyuiihjjrftg}
\end{equation}
where a positive  constant $C$ depends on $p$, 
${q}$ and $\delta$.
\end{lemma}

\begin{proof}
Firstly, applying
  the  H\"{o}lder inequality  leads
\begin{equation}
\begin{array}{rl}
J_3&:=\disp{ \int_\Omega |Du{}_{\varepsilon}| |\nabla c_{\varepsilon}|^{2{q}}}\\
&\leq\disp{\|Du_{\varepsilon}\|_{L^{q+1}(\Omega)}
 \|\nabla c_{\varepsilon}\|_{L^{2(q+1)}(\Omega)}^{2{q}}.}\\
\end{array}
\label{22cz2.57151hhkkhhhjukildrfttyuijkoghyu66}
\end{equation}
Due to Lemma \ref{lemma630jklhhjj}, there exists a positive constant $C_1$ such that
\begin{equation}\|Du_{\varepsilon}\|_{L^{l}(\Omega)}\leq C_{1},
\label{22cz2.57151hhkkhhhjukildrfttyuifffgggjkoghyu66}
\end{equation}
where $l<\frac{3}{2}$ is the  same as Lemma \ref{lemma630jklhhjj}.
Hence, using the Gagliardo--Nirenberg inequality and \dref{22cz2.57151hhkkhhhjukildrfttyuifffgggjkoghyu66}, since $q +1\leq\frac{2l+3}{3} +1 < 6$ we can find
$C_{1} $, $C_{2} $ and $C_{3} $ such that
\begin{equation}
\begin{array}{rl}
&\disp{\|Du_{\varepsilon}\|_{L^{q+1}(\Omega)}}
\\
\leq&\disp{C_{1}\|u_{\varepsilon}\|_{W^{2,2}(\Omega)}^{\frac{6(q+1-l)}{(q+1)(6-l)}}\|u_{\varepsilon}\|_{W^{1,l}(\Omega)}^{\frac{(5-q)l}{(q+1)(6-l)}}}
\\
\leq&\disp{C_{2}\|Au_{\varepsilon}\|_{L^{2}(\Omega)}^{\frac{6(q+1-l)}{(q+1)(6-l)}}\|Du_{\varepsilon}\|_{L^{l}(\Omega)}^{\frac{(5-q)l}{(q+1)(6-l)}}}
\\
\leq&\disp{C_{3}\|Au_{\varepsilon}\|_{L^{2}(\Omega)}^{\frac{6(q+1-l)}{(q+1)(6-l)}}.}
\\
\end{array}
\label{22cz2.563022222ikopl2sdfgggjjkkk66}
\end{equation}

Now, in light of  Lemma \ref{ghjssdeedrfe116lemma70hhjj} and
the Gagliardo--Nirenberg inequality that
we have
\begin{equation}
\begin{array}{rl}
 \|\nabla {c_{\varepsilon}}\|_{L^{2(q+1)}(\Omega)}^{2{q}}\leq&\disp{C_{4}(\||\nabla{c_{\varepsilon}}|^{q-1}D^2{c{}}\|_{L^2(\Omega)}
 ^{\frac{2q[2(2q+2)-6]}{(2q-1)(2q+2)}}
 \|{c_{\varepsilon}}\|_{L^\infty(\Omega)}^{\frac{2q[6q-2(q+1)]}{(2q-1)(2q+2)}}+\|{c_{\varepsilon}}\|_{L^\infty(\Omega)}^{2q})}
\\
\leq&\disp{C_{5}(\||\nabla{c_{\varepsilon}}|^{q-1}D^2{c_{\varepsilon}}\|_{L^2(\Omega)}^{\frac{2q}{q+1}}+1),}\\
\end{array}
\label{22czffsdefr2.563022222ikopl255}
\end{equation}
where 
 $C_{4}$ and  $C_{5}$ are  positive constants independent of $\varepsilon$.
Inserting \dref{cz2.563022222ikopl2sdfgggjjkkk66}--\dref{czsdefr2.563022222ikopl255} into \dref{cz2.57151hhkkhhhjukildrfttyuijkoghyu66} and using $q < \frac{2l+3}{3}$ and the Young inequality, we have
\begin{equation}
\begin{array}{rl}
J_3\leq&\disp{C_{6}\|Au_{\varepsilon}\|_{L^{2}(\Omega)}^{\frac{6(q+1-l)}{(q+1)(6-l)}}
(\||\nabla{c_{\varepsilon}}|^{q-1}D^2{c_{\varepsilon}}\|_{L^2(\Omega)}^{\frac{2q}{q+1}}+1)}\\
\leq&\disp{\delta\int_{\Omega} |Au_{\varepsilon}|^2+\delta\||\nabla{c_{\varepsilon}}|^{q-1}D^2{c_{\varepsilon}}\|_{L^2(\Omega)}^2
+C_{7}~~\mbox{for all}~~ t\in(0,T).  }\\
\end{array}
\label{223czff2.57151hhkkhhhjukildrftdfrtgyhugg}
\end{equation}
with some positive constants $C_6$ and $C_7.$
\end{proof}
\begin{lemma}\label{lemma4ggg563025xxhjklojjkkkgyhuirrtt}
Assuming that
${q}>1$.
If \begin{equation}
\|D u_{\varepsilon}(\cdot, t)\|_{L^{2}(\Omega)}\leq K~~ \mbox{for all}~~ t\in(0, T),\label{ddffcz2.5715cgghhcgghhvv1hhkkhhggjjllll}
\end{equation}
then for all
small $\delta>0$, the solution of \dref{1.1fghyuisda} from Lemma \ref{lemma70} satisfies
\begin{equation}
\begin{array}{rl}
&\disp{\int_\Omega |Du{}_{\varepsilon}||\nabla c_{\varepsilon}{}|^{2{q}}\leq
\delta\|\nabla |\nabla c_{\varepsilon}|^{q}\|_{L^2(\Omega)}^2+C ~~ \mbox{for all}~~ t\in(0, T), }\\
\end{array}
\label{cz2.5ghhjffddcvgggghhuyuigggihjjrftg}
\end{equation}
where a positive  constant $C$ depends on  
${q}$ and $\delta$.
\end{lemma}
\begin{proof}
Firstly, using the  H\"{o}lder inequality,
 we find
\begin{equation}
\begin{array}{rl}
J_3&:=\disp{ \int_\Omega |Du{}_{\varepsilon}| |\nabla c_{\varepsilon}|^{2{q}}}\\
&\leq\disp{\|Du_{\varepsilon}\|_{L^{2}(\Omega)}
 \|\nabla c_{\varepsilon}\|_{L^{4q}(\Omega)}^{2{q}}}\\
 &\leq\disp{K
 \|\nabla c_{\varepsilon}\|_{L^{4q}(\Omega)}^{2{q}}~~ \mbox{for all}~~ t\in(0, T).}\\
\end{array}
\label{cz2.57151hhkkhhhjukildddrfttyuijkoghyu66}
\end{equation}
Now, it then follows from  Lemma \ref{ghjssdeedrfe116lemma70hhjj} and
the Gagliardo--Nirenberg inequality that
we have
\begin{equation}
\begin{array}{rl}
 \|\nabla {c_{\varepsilon}}\|_{L^{4q}(\Omega)}^{2{q}}\leq&\disp{C_{1}(\||\nabla{c_{\varepsilon}}|^{q-1}D^2{c{}}\|_{L^2(\Omega)}
 ^{\frac{4q[4q-3]}{(2q-1)4q}}
 \|{c_{\varepsilon}}\|_{L^\infty(\Omega)}^{\frac{2q[6q-4q]}{(2q-1)4q}}+\|{c_{\varepsilon}}\|_{L^\infty(\Omega)}^{2q})}
\\
\leq&\disp{C_{2}(\||\nabla{c_{\varepsilon}}|^{q-1}D^2{c_{\varepsilon}}\|_{L^2(\Omega)}^{\frac{2q-3}{2q-1}}+1),}\\
\end{array}
\label{czffsdefr2.5630222ddff22ikopl255}
\end{equation}
where 
some positive constants $C_{1}, C_{2}$.
Inserting \dref{czffsdefr2.5630222ddff22ikopl255} into \dref{cz2.57151hhkkhhhjukildddrfttyuijkoghyu66} and employing the Young inequality, we can get \dref{cz2.5ghhjffddcvgggghhuyuigggihjjrftg}.
\end{proof}

\begin{lemma}\label{lemma4563025xxhjklojjkkkgyhuissddff}
Assuming that $m>\frac{9}{8}$.
Then for all $p>1$ and ${q}>1$,
\begin{equation}\|{n}_\varepsilon(\cdot, t)\|_{L^p(\Omega)}+\|\nabla {c}_{\varepsilon}(\cdot, t)\|_{L^{2{q}}(\Omega)}\leq C~~ \mbox{for all}~~ t\in(0, T), \label{ddffcz2.5715ccgghhvv1hhkkhvvbnnhggjjllll}
\end{equation}
 where $T$ is given by \dref{hjnmkcz2.5ffvvbbghhjuyuiihjjrftg}.
\end{lemma}
\begin{proof}
We only need to prove case $2\geq m>\frac{9}{8}$.
Since $m>2$, employing almost exactly the same arguments as in the proof of Lemma 4.1 in \cite{Zhengsdsd6} (see also \cite{Winkler11215})
(the minor necessary changes are left as an easy exercise to the reader), 
we conclude the estimate
\dref{ddffcz2.5715ccgghhvv1hhkkhvvbnnhggjjllll}.
%
%
%
%
%
%
%

Case $2\geq m>\frac{9}{8}$.
We divide the proof into two steps.

Step 1. We first make sure that there exists $C_{1}>0$ such that
\begin{equation}
\|D u_{\varepsilon}(\cdot, t)\|_{L^{2}(\Omega)}\leq C_{1}~~ \mbox{for all}~~ t\in(0, T).\label{ddffcz2.5715ccgghhvv1hhkkhhggjjllll}
\end{equation}

To achieve this,
in light of $m>\frac{9}{8},$ $p_0\geq1$ and $q>1$, then
\begin{equation}
\begin{array}{rl}
[2(m-1)+\frac{2p_0}{3}]q+m-1>&3(m-1)+\frac{2p_0}{3}\\
>&3\times\frac{1}{9}+\frac{2}{3}=1,
\end{array}
\label{cz2.57151hnnnnhkffgssdffhhkhhggjjllll}
\end{equation}
\begin{equation}
\begin{array}{rl}
[2(m-1)+\frac{2p_0}{3}]q+m-1>&m-1+\frac{2p_0}{3}q\\
>&m-1+p_0\frac{q}{q+1}
\end{array}
\label{cz2.57151hnnvvnnhkffgssdffhhkhhggjjllll}
\end{equation}
and
\begin{equation}
\begin{array}{rl}
\left\{[2(m-1)+\frac{2p_0}{3}]q+m-1\right\}-\left\{\frac{q+1}{3}+1-m\right\}=&2(m-1)(q-1)+(\frac{2p_0}{3}-\frac{1}{3})q\\
>&\frac{1}{3}q>0.
\end{array}
\label{cz2.57151hnnccvvnnhkffgssdffhhkhhggjjllll}
\end{equation}
Now, since $p_0\in [1,9(m-1))$ is arbitrary, with the help of $m>\frac{9}{8}$, we can finally pick $p_0<9(m-1)$
 sufficiently close
to
$9(m-1)$ such that
$$
2(m-1)+\frac{2p_0}{3}-1>0,
$$
which implies that for all $q>1$ and $p_0<9(m-1)$ sufficiently close $9(m-1)$
to
\begin{equation}
\begin{array}{rl}
q+2-\frac{2p_0}{3}-m<[2(m-1)+\frac{2p_0}{3}]q+m-1,
\end{array}
\label{cz2.57151hhkffgccvvvvssdffhhkhhggjjllll}
\end{equation}
which along with \dref{cz2.57151hnnnnhkffgssdffhhkhhggjjllll}--\dref{cz2.57151hnnccvvnnhkffgssdffhhkhhggjjllll} yields that
\begin{equation}
\begin{array}{rl}
&\max\{1,m-1+p_0\frac{q}{q+1},\frac{q+1}{3}+1-m,q+2-\frac{2p_0}{3}-m\}\\
<&[2(m-1)+\frac{2p_0}{3}]q+m-1
~~\mbox{for all}~~q>1.
\end{array}
\label{cz2.57151hhkffgssdffhhkhhggjjllll}
\end{equation}
Now, in view of $l<\frac{3}{2}$ is arbitrary, we can finally pick $q_0<2$
 sufficiently close
to
$2$ such that $q_0<\frac{2l+3}{3}$.
 Now, we choose $q=q_0$ in \dref{cz2.57151hhkffgssdffhhkhhggjjllll}, $$p\in(\max\{1,m-1+p_0\frac{q_0}{q_0+1},\frac{q_0+1}{3}+1-m,q_0+2-\frac{2p_0}{3}-m\},[2(m-1)+\frac{2p_0}{3}]q_0+m-1)$$ and $\delta$ small enough in Lemma \ref{ffglffhhemmaccvbbfggg78630jklhhjj}--Lemma \ref{lemma4563025xxhjklojjkkkgyhuirrtt},
then by Lemma \ref{lemmaghjmk4563025xxhjklojjkkk}, we conclude that
\begin{equation}
\begin{array}{rl}
&\disp{\frac{d}{dt}\left(\int_{\Omega}n^{p}_{\varepsilon} +\int_{\Omega} |\nabla {c}_{\varepsilon}|^{2{q_0}}+\int_{\Omega}|A^{\frac{1}{2}}u_{\varepsilon}|^2 \right)+\frac{(p-1)}{4C_D}\int_{\Omega}n_{\varepsilon}^{m+p-3} |\nabla n_{\varepsilon}|^2}\\
&+\disp{\frac{({q_0}-1)}{{{q_0}^2}}\int_{\Omega}\left|\nabla |\nabla {c}_{\varepsilon}|^{{q_0}}\right|^2+\frac{1}{4}\int_\Omega  |\nabla {c}_{\varepsilon}|^{2{q_0}-2}|D^2c_{\varepsilon}|^2+\int_{\Omega}|Au_{\varepsilon}|^2}\\
\leq&\disp{C_{2}}\\
\end{array}
\label{cz2.57151hhkkhhggjjllll}
\end{equation}
with some positive constant $C_2.$
Assuming that  $y:=\int_{\Omega}n^{p}_{\varepsilon} +\int_{\Omega} |\nabla {c}_{\varepsilon}|^{2{q_0}}+\int_{\Omega}|A^{\frac{1}{2}}u_{\varepsilon}|^2 $,
then \dref{cz2.57151hhkkhhggjjllll} implies that there exist positive constants $C_3$ and $C_4$ such that
$$\frac{d}{dt}y(t)+C_{3}y^h(t)\leq C_{4},$$
where $h$ is a positive constant.
Then an ODE comparison argument yields  that there exists a positive constant $C_5$ independent of $\varepsilon$ such that
\begin{equation}
\|n_{\varepsilon}(\cdot, t)\|_{L^p(\Omega)} +\|\nabla {c}_{\varepsilon}(\cdot, t)\|_{L^{2{q_0}}(\Omega)}+\|A^{\frac{1}{2}}u_{\varepsilon}(\cdot, t)\|_{L^2(\Omega)}\leq C_{5}~~ \mbox{for all}~~ t\in(0, T),\label{cz2.5715ccvv1hhkkhhggjjllll}
\end{equation}
which together with $q_0<2$, $p<[2(m-1)+\frac{2p_0}{3}]q_0+m-1$, $p_0<9(m-1)$
 sufficiently close
to
$9(m-1)$ and $2\geq m>\frac{9}{8}$ imply that
\begin{equation}
\|n_{\varepsilon}(\cdot, t)\|_{L^{\frac{7}{4}}(\Omega)}\leq C_{6}~~ \mbox{for all}~~ t\in(0, T).\label{cz2.5715ccgghhvv1hhkkhhggjjllll}
\end{equation}
with some positive constant $C_6$.
Therefore, by Lemma \ref{lemmafggg78630jklhhjj}, we can get \dref{ddffcz2.5715ccgghhvv1hhkkhhggjjllll}.

Step 2. We proceed to prove the statement of the lemma.
To this end, by \dref{ddffcz2.5715ccgghhvv1hhkkhhggjjllll}, we can get \dref{ddffcz2.5715cgghhcgghhvv1hhkkhhggjjllll}, hence,
\dref{cz2.5ghhjffddcvgggghhuyuigggihjjrftg} holds.
Now, inserting \dref{cz2.5ghhjuyuiihddfffjjrftg}, \dref{cz2.5ghhjuyuiihjjhhhhrftg},  \dref{cz2.5ghhjuyuiihjjrftg} and \dref{cz2.5ghhjffddcvgggghhuyuigggihjjrftg} into \dref{cz2.5ghhjuyuiihjj} and using \dref{cz2.57151hhkffgssdffhhkhhggjjllll} yields to
\begin{equation}
\begin{array}{rl}
&\disp{\frac{d}{dt}\left(\int_{\Omega}n^{p}_{\varepsilon} +\int_{\Omega} |\nabla {c}_{\varepsilon}|^{2{q}} \right)+\frac{(p-1)}{4C_D}\int_{\Omega}n_{\varepsilon}^{m+p-3} |\nabla n_{\varepsilon}|^2}\\
&+\disp{\frac{({q}-1)}{{{q}^2}}\int_{\Omega}\left|\nabla |\nabla {c}_{\varepsilon}|^{{q}}\right|^2+\frac{1}{4}\int_\Omega  |\nabla {c}_{\varepsilon}|^{2{q}-2}|D^2c_{\varepsilon}|^2}\\
\leq&\disp{C_{7},}\\
\end{array}
\label{cz2.57151hhjjjkkhhggjjllll}
\end{equation}
where $C_7$ is a positive constant independent of $\varepsilon$.
Here we have picked the $\delta$ small enough. Employing  the same arguments as in the proof of step 1,
we conclude \dref{ddffcz2.5715ccgghhvv1hhkkhvvbnnhggjjllll}.
The proof of Lemma \ref{lemma4563025xxhjklojjkkkgyhuissddff} is complete.
\end{proof}

\section{The proof of main results}

In preparation of an Aubin-Lions type compactness argument, we intend to supplement Lemma \ref{lemma4563025xxhjklojjkkkgyhuissddff}
with bounds on time-derivatives.
With these estimates, we can construct weak solutions by   means of a standard extraction procedure, therefore,  we only give general idea, one can see \cite{Zhengsdsd6,Winkler11215} for more details.

{\bf The proof of Theorem \ref{theorem3}}~
Firstly, choosing $p=2$ in \dref{ddffcz2.5715ccgghhvv1hhkkhvvbnnhggjjllll}, we derive that
\begin{equation}\|n_{\varepsilon}(\cdot,t)\|_{L^2(\Omega)} \leq C_1~~\mbox{for all}~~ t\in(0,T)
\label{ggzjscz2.5297x96dd2234530111kkhhffrreerr}
\end{equation}
with some positive constant $C_1.$
Now, in view of \dref{ggzjscz2.5297x96dd2234530111kkhhffrreerr}, employing  the variation-of-constants formula for $u_{\varepsilon}$ and by the properties of the
Stokes semigroup, we conclude that there exists a positive constant $C_2$ such that
\begin{equation}
\|A^\gamma u_{\varepsilon}(\cdot,t)\|_{L^{2}(\Omega)}  \leq C_2~~\mbox{for all}~~ t\in(0,T),
\label{ggzjscz2.5297x9630111kkhhffrreerr}
\end{equation}
where $\gamma\in(\frac{3}{4},1)$ and $T$ are the same as \dref{ccvvx1.731426677gg} and \dref{hjnmkcz2.5ffvvbbghhjuyuiihjjrftg}, respectively.
Since, $D(A^\gamma)$ is continuously embedded into $L^\infty(\Omega)$, hence, in view of \dref{ggzjscz2.5297x9630111kkhhffrreerr}, we may find that
 \begin{equation}
\begin{array}{rl}
\|u_\varepsilon(\cdot, t)\|_{L^\infty(\Omega)}\leq C_3~~ \mbox{for all}~~ t\in(0,T).\\
\end{array}
\label{cz2.571ttgyhhuuy5jkkcvccvvhjjjkddfffffkhhgll}
\end{equation}
for some positive constant $C_3.$
Next, choosing $q=2$ in \dref{ddffcz2.5715ccgghhvv1hhkkhvvbnnhggjjllll}, we get that
\begin{equation}\|\nabla c_{\varepsilon}(\cdot,t)\|_{L^{4}(\Omega)} \leq C_4~~\mbox{for all}~~ t\in(0,T),
\label{ggzjscz2.529uuu7x96dd2234530111kkhhffrreerr}
\end{equation}
where $C_4$ is positive constant independent of $\varepsilon.$
Hence, in light of  \dref{cz2.571ttgyhhuuy5jkkcvccvvhjjjkddfffffkhhgll}--\dref{ggzjscz2.529uuu7x96dd2234530111kkhhffrreerr} and the $L^p$-$L^q$ estimates associated heat semigroup, we derive that there exists a positive constant $C_5$
such that
\begin{equation}
\|c_{\varepsilon}(\cdot,t)\|_{W^{1,\infty}(\Omega)}  \leq C_5 ~~\mbox{for all}~~ t\in(0,T).
\label{zjscz2.5297x9630111kkhh}
\end{equation}
Now, due to \dref{cz2.571ttgyhhuuy5jkkcvccvvhjjjkddfffffkhhgll} and \dref{zjscz2.5297x9630111kkhh}, we may use the standard
Moser-type iteration to conclude that there exists a positive constant $C_6$ such that
 \begin{equation}
\|n_{\varepsilon}(\cdot,t)\|_{L^\infty(\Omega)}  \leq C_6 ~~\mbox{for all}~~ t\in(0,T).
\label{zjscz2.5297x9630111kk}
\end{equation}
%

Next,
  with the help of \dref{cz2.571ttgyhhuuy5jkkcvccvvhjjjkddfffffkhhgll}, \dref{zjscz2.5297x9630111kkhh} and \dref{zjscz2.5297x9630111kk}, for all $\varepsilon\in(0,1),$ we can fix a positive constants $C_7$ such that
 \begin{equation}
n_\varepsilon\leq C_7,|\nabla c_\varepsilon|  \leq C_7~~\mbox{and}~~|u_\varepsilon|  \leq C_7 ~~\mbox{in}~~ \Omega\times(0,T).
\label{gbhnzjscz2.5297x9630111kkhhiioo}
\end{equation}
Now, employing the same argument of Lemma 5.1 of \cite{Zhengsdsd6} (see also Lemmata 3.22 and 3.23  of \cite{Winkler11215}), we may derive that
there exists a positive constant $C_8$ such that
\begin{equation}
\|\partial_tn_\varepsilon(\cdot,t)\|_{(W^{2,2}_0(\Omega))^*}  \leq C_8 ~~\mbox{for all}~~ t>0
\label{hjjjzjscz2.5297x9630111kkhhiioott}
\end{equation}
as well as
\begin{equation}
\int_0^T\|\partial_tn_\varepsilon^\varsigma(\cdot,t)\|_{(W^{3,2}_0(\Omega))^*}dt  \leq C_8(T+1)
\label{zkkllljscz2.5297x9630111kkhhiioott4}
\end{equation}
and
 \begin{equation}
\int_{0}^T\int_{\Omega}n_{\varepsilon}^{m+p-3} |\nabla n_{\varepsilon}|^2\leq C_8T,
\label{fvgbhzjscz2.5297x96302222tt4455hyuhii}
 \end{equation}
 where $p := 2\zeta-m+1$ and $\varsigma> m$ satisfying $\varsigma\geq2(m-1)$.

Now, in conjunction with \dref{ggzjscz2.5297x9630111kkhhffrreerr}, \dref{cz2.571ttgyhhuuy5jkkcvccvvhjjjkddfffffkhhgll}, \dref{zjscz2.5297x9630111kkhh} and \dref{zjscz2.5297x9630111kk}
 and the Aubin--Lions compactness lemma (\cite{Simon}), we thus
infer that there exists a sequence of numbers $\varepsilon: = \varepsilon_j \searrow0$ such that
\begin{equation}
 n_\varepsilon\rightharpoonup n ~~\mbox{weakly star in}~~ L^\infty_{loc}(\Omega\times(0,\infty)),
 \label{ffzjscz2.5297x9630222222ee}
\end{equation}
\begin{equation}
n_\varepsilon\rightarrow n ~~\mbox{in}~~ L^\infty_{loc}([0,\infty); (W^{3,2}_0 (\Omega))^*),
\label{ffzjscz2.5297x96302222tt}
\end{equation}
\begin{equation}
c_\varepsilon\rightarrow c ~~~ ~~\mbox{a.e.}~~ \mbox{in}~~ \Omega\times (0,\infty),
 \label{ffzjscz2.5297x96302222tt3}
\end{equation}
\begin{equation}
\nabla c_\varepsilon\rightarrow \nabla c ~~\mbox{in}~~ L^\infty_{loc}(\bar{\Omega}\times[0,\infty)),
 \label{ffzjscz2.5297x96302222tt4}
\end{equation}
\begin{equation}
u_\varepsilon\rightarrow u ~~\mbox{a.e.}~~ \mbox{in}~~ \Omega\times (0,\infty),
 \label{ffzjscz2.5297x96302222tt44}
\end{equation}
and
\begin{equation}
D u_\varepsilon\rightharpoonup Du ~~\mbox{weakly in}~~L^{\infty}_{loc}(\Omega\times[0,\infty))
 \label{ffzjscz2.5297x96302222tt4455}
\end{equation}
holds for some triple  $(n, c, u)$.

%
Since $p = 2\zeta-m+1$, then by \dref{fvgbhzjscz2.5297x96302222tt4455hyuhii} implies
that
for each
$T > 0,$ $(n_{\varepsilon}^\varsigma)_{\varepsilon\in(0,1)}$ is bounded in $L^2((0, T);W^{1,2}(\Omega))$.
 With the help of \dref{zkkllljscz2.5297x9630111kkhhiioott4}, we  also show that
$$(\partial_{t}n_{\varepsilon}^\varsigma)_{\varepsilon\in(0,1)}~~\mbox{is bounded in}~~L^1((0, T); (W^{3,2}_0(\Omega))^*)~~\mbox{for each}~~  T > 0.$$
Hence, an Aubin-Lions lemma (see e.g. \cite{Simon})
applies to the above inequality we have the strong precompactness of
%
%
%
$(n_{\varepsilon}^\varsigma)_{\varepsilon\in(0,1)}$ in $L^2(\Omega\times(0, T))$.
Therefore,
we can
  pick a
suitable subsequence such that $n_{\varepsilon}^\varsigma\rightarrow z^\varsigma$ for some nonnegative measurable $z:\Omega\times(0,\Omega)\rightarrow\mathbb{R}$.
 In
light of  \dref{ffzjscz2.5297x9630222222ee} and the Egorov theorem, we have $z = n$ necessarily, so that
\begin{equation} n_\varepsilon\rightarrow n ~~\mbox{a.e.}~~ \mbox{in}~~ \Omega\times (0,\infty)
\label{ffzjscz2.5297x9630222222ghjj}
\end{equation}
is valid.
Now, in light of \dref{ffzjscz2.5297x9630222222ee},
 \dref{ffzjscz2.5297x96302222tt4},
and \dref{ffzjscz2.5297x96302222tt4455}, we derive that \dref{zj233455scz2.529711234566x9630111kk} is hold.

Now, we will prove that $(n, c,u)$ is the global weak solution of \dref{1.1}.
To this end, by \dref{ffzjscz2.5297x9630222222ghjj}, \dref{ffzjscz2.5297x9630222222ee}, \dref{ffzjscz2.5297x96302222tt3}, \dref{ffzjscz2.5297x96302222tt4}
and \dref{ffzjscz2.5297x96302222tt44}, we have
$n$, $c$ are nonnegative  and \dref{dffff1.1fghyuisdakkklll}
 and \dref{726291hh} are valid.
Next, with the help of the standard arguments,  letting $\varepsilon = \varepsilon_j \searrow0$ in the approximate system \dref{1.1fghyuisda} and using
\dref{ffzjscz2.5297x9630222222ghjj}--\dref{ffzjscz2.5297x9630222222ee} and
\dref{ffzjscz2.5297x96302222tt3}--\dref{ffzjscz2.5297x96302222tt4455}, we can get \dref{eqx45xx12112ccgghh}--\dref{eqx45xx12112ccgghhjjgghh}.
%

{\bf Acknowledgement}:
This work is partially supported by  the National Natural Science Foundation of China (No. 11601215),
the
Natural Science Foundation of Shandong Province of China (No. ZR2016AQ17) and the Doctor Start-up Funding of Ludong University (No. LA2016006).

\end{document}